\documentclass[10pt]{article}
\pdfoutput=1

\usepackage[utf8]{inputenc}
\usepackage[T1]{fontenc}
\usepackage{lmodern}

\usepackage{amsmath}
\usepackage{amssymb}
\usepackage{amsthm}
\usepackage{mathtools}
\usepackage{dsfont}
\usepackage{aliascnt}

\usepackage{geometry}
\usepackage{enumitem}
\usepackage{xcolor}

\usepackage{natbib}
\usepackage[unicode]{hyperref}
\usepackage{cleveref}

\setcitestyle{numbers,open={[},close={]}}

\definecolor{linkcol}{rgb}{0.7,0.15,0.15}
\definecolor{citecol}{rgb}{0,0.5,0}
\definecolor{urlcol}{rgb}{0,0,0.7}
\hypersetup{colorlinks, linkcolor=linkcol, citecolor=citecol, urlcolor=urlcol}

\numberwithin{equation}{section}

\DeclareUnicodeCharacter{014D}{\=o}
\newtheorem{theorem}{Theorem}[section]
\crefname{theorem}{theorem}{theorems}
\Crefname{theorem}{Theorem}{Theorems}

\newaliascnt{assumption}{theorem}
\newtheorem{assumption}[assumption]{Assumption}
\aliascntresetthe{assumption}
\crefname{assumption}{assumption}{assumptions}
\Crefname{assumption}{Assumption}{Assumptions}

\newaliascnt{corollary}{theorem}
\newtheorem{corollary}[corollary]{Corollary}
\aliascntresetthe{corollary}
\crefname{corollary}{corollary}{corollaries}
\Crefname{corollary}{Corollary}{Corollaries}

\newaliascnt{example}{theorem}

\aliascntresetthe{example}
\crefname{example}{example}{examples}
\Crefname{example}{Example}{Examples}

\newaliascnt{exercise}{theorem}

\aliascntresetthe{exercise}
\crefname{exercise}{exercise}{exercises}
\Crefname{exercise}{Exercise}{Exercises}

\newaliascnt{lemma}{theorem}

\aliascntresetthe{lemma}
\crefname{lemma}{lemma}{lemmas}
\Crefname{lemma}{Lemma}{Lemmas}

\newaliascnt{proposition}{theorem}
\newtheorem{proposition}[proposition]{Proposition}
\aliascntresetthe{proposition}
\crefname{proposition}{proposition}{propositions}
\Crefname{proposition}{Proposition}{Propositions}

\newaliascnt{condition}{theorem}

\aliascntresetthe{condition}
\crefname{condition}{condition}{conditions}
\Crefname{condition}{Condition}{Conditions}

\newaliascnt{definition}{theorem}
\newtheorem{definition}[definition]{Definition}
\aliascntresetthe{definition}
\crefname{definition}{definition}{definitions}
\Crefname{definition}{Definition}{Definitions}

\newaliascnt{remark}{theorem}
\newtheorem{remark}[remark]{Remark}
\aliascntresetthe{remark}
\crefname{remark}{remark}{remarks}
\Crefname{remark}{Remark}{Remarks}

\def\no{\noindent}

\def\E{\mathbb{E}}
\def\F{\mathbb{F}}
\def\P{\mathbb{P}}
\def\Q{\mathbb{Q}}
\def\R{\mathbb{R}}

\def\Cc{{\cal C}}
\def\Fc{{\cal F}}
\def\Gc{{\cal G}}
\def\Kc{{\cal K}}
\def\Lc{{\cal L}}
\def\Pc{{\cal P}}
\def\Uc{{\cal U}}
\def\Vc{{\cal V}}
\def\Zc{{\cal Z}}

\def\Gr{{\rm G}}
\def\Ir{{\rm I}}
\def\Kr{{\rm K}}
\def\trace{{\rm Tr}}
\def\Lgen{\mathfrak{L}}

\def\Sbb{\mathbf{S}}
\def\Xbb{\mathbf{X}}
\def\Ubb{\boldsymbol{U}}
\def\xbb{\boldsymbol{x}}

\def\1{\mathds{1}}

 \title{A McKean--Vlasov semigroup and its application to a martingale representation problem}

\author{
 Mao Fabrice Djete\footnote{\'Ecole Polytechnique Paris, Centre de Math\'ematiques Appliqu\'ees, mao-fabrice.djete@polytechnique.edu. The author benefits from the financial support of the Chairs {\it Financial Risk} and {\it Finance and Sustainable Development}} 
 \and 
 Mattia Martini \footnote{\'Ecole Polytechnique Paris, Centre de Math\'ematiques Appliqu\'ees, mattia.martini@polytechnique.edu. The author thanks the support of the French ANR PEPR Math-Vives project MIRTE ANR-23-EXMA-0011}
   }
   
\date{\today}

\begin{document}

\maketitle
 
\begin{abstract}
We introduce a McKean--Vlasov semigroup in the presence of common noise and study its first-order differential structure. Given a terminal functional $G$, the semigroup is defined by propagating the initial law through a reference conditional McKean--Vlasov flow. Using a finite-particle approximation and a diagonal/off-diagonal decomposition of the tangent processes, we identify the Lions derivative of the propagated functional and derive a backward representation along arbitrary admissible conditional law flows. This formula involves only first-order derivatives in the measure variable and yields both a Feynman--Kac representation for linear equations on the Wasserstein space and, in the driftless case, a martingale representation problem whose unique solution is given by the McKean--Vlasov heat semigroup and its Lions derivative. We also discuss mixed finite-dimensional/Wasserstein representations and discrete path-dependent extensions.
\end{abstract}

\vspace{3mm}
\no{\bf Keywords.} McKean--Vlasov semigroup, common noise, martingale representation.

\vspace{3mm}
\no{\bf MSC2020.} 60H30, 60K35, 49N80, 47D07, 60G44, 35R15.

\section{Introduction}
The heat equation and the associated heat semigroup are classical objects in
the theory of parabolic partial differential equations. Standard references
include \citeauthor*{friedman1964partial} \cite{friedman1964partial}, \citeauthor*{ladyzhenskaya1968linear} \cite{ladyzhenskaya1968linear}, and \citeauthor*{evans2010partial} \cite{evans2010partial}. From
the semigroup point of view, the heat flow is the prototype of a strongly
continuous Markov semigroup generated by the Laplacian; see for instance
\citeauthor{pazy1983semigroups} \cite{pazy1983semigroups}, \citeauthor*{engel2000one} \cite{engel2000one}, and
\citeauthor*{ethier1986markov} \cite{ethier1986markov}. Its probabilistic interpretation through
Brownian motion, Itô's formula and the Feynman--Kac formula is classical; see
\citeauthor*{karatzas1991brownian} \cite{karatzas1991brownian}, \citeauthor*{revuz1999continuous} \cite{revuz1999continuous},
and \citeauthor*{stroock1979multidimensional} \cite{stroock1979multidimensional}. Let us recall this probabilistic interpretation and highlight the points we want to emphasize. Let
$u:[0,T]\times\R^d\to\R$ be smooth, say $u\in C^{1,2}$, and let
\[
X_t=X_0+W_t+\sigma_\circ W^\circ_t,
\]
where $W$ and $W^\circ$ are independent Brownian motions and
$\sigma_\circ\in\R^{d\times d}$. By Itô's formula, for every $0\le t\le T$,
\[
\begin{aligned}
u(t,X_t)
=
u(T,X_T)
-\int_t^T \nabla_x u(s,X_s)\cdot \mathrm{d}X_s
-\int_t^T
\left(
\partial_t u(s,X_s)
+
\frac12
\trace\left[
\nabla_x^2u(s,X_s)
\bigl(\Ir_d+\sigma_\circ\sigma_\circ^\top\bigr)
\right]
\right)\mathrm{d}s .
\end{aligned}
\]
Hence, if $u$ solves the backward heat equation
\[
\partial_t u(t,x)
+
\frac12
\trace \left[
\nabla_x^2u(t,x)
\bigl(\Ir_d+\sigma_\circ\sigma_\circ^\top\bigr)
\right]
=0,
\qquad (t,x)\in[0,T)\times\R^d,
\]
then the finite-variation term disappears and one obtains
\[
u(t,X_t)
=
u(T,X_T)
-
\int_t^T
\nabla_xu(s,X_s)\cdot \mathrm{d}X_s.
\]
Therefore, by taking the conditional expectation,
\[
u(t,x)
=
\Kr_{u(T,\cdot)}(t,x)
:=
\E\bigl[
u(T,x+W_T-W_t+\sigma_\circ(W^\circ_T-W^\circ_t))
\bigr].
\]
The operator $\Kr$ is the heat semigroup associated with the covariance $\Ir_d+\sigma_\circ\sigma_\circ^\top$.
Thus $(u(t,X_t))_{0\le t\le T}$ is a martingale with respect to the filtration
generated by $X$. The important point is that the backward representation only
contains the first spatial derivative $\nabla_xu$. The time derivative
$\partial_tu$ and the second--order derivative $\nabla_x^2u$ have been
absorbed by the heat equation.

\medskip
Our main interest is the converse point of view. Suppose that a pair of maps
$(u,v)$ satisfies
\begin{align} \label{eq:martingale}
u(t,X_t)
=
u(T,X_T)
-
\int_t^T
v(s,X_s)\cdot \mathrm{d}X_s .
\end{align}
In this Markovian setting, this is another formulation of the martingale
representation theorem when the map $u(T,\cdot)$ is given. Under suitable integrability assumptions on $v$,
taking conditional expectation gives
\[
u(t,X_t)
=
\E\bigl[u(T,X_T)\mid \Fc_t\bigr]
=
\Kr_{u(T,\cdot)}(t,X_t).
\]
By identification of the martingale integrand, one then obtains
\[
v(t,X_t)
=
\nabla_x\Kr_{u(T,\cdot)}(t,X_t).
\]
Therefore, the heat semigroup not only provides examples of pairs satisfying
\eqref{eq:martingale}; it also characterizes them. It is worth mentioning that no time
derivative of $u$ and no second--order derivative in space are required in
the formulation of the representation. The heat semigroup is precisely the
object which produces, and under natural conditions uniquely produces, such a
first-order martingale dynamics.

\medskip
The same picture holds for a diffusion with drift. If
\[
\mathrm{d}X_t
=
B(t,X_t)\,\mathrm{d}t
+
\mathrm{d}W_t+\sigma_\circ\,\mathrm{d}W^\circ_t,
\]
then, by a Girsanov transformation, the martingale identity for the Brownian
dynamics can be transported to the diffusion $X$. Along the trajectory of
$X$, the pair
\[
\bigl(\Kr_{u(T,\cdot)},\nabla_x\Kr_{u(T,\cdot)}\bigr)
\]
is characterized by the identity
\begin{align} \label{eq:martingale_drift}
u(t,X_t)
=
u(T,X_T)
-
\int_t^T
v(s,X_s)\cdot B(s,X_s)\,\mathrm{d}s
-
\int_t^T
v(s,X_s)\cdot
\bigl(\mathrm{d}W_s+\sigma_\circ\,\mathrm{d}W^\circ_s\bigr).
\end{align}
This is the Markovian form of the martingale representation principle and is
one of the structural reasons why classical Backward Stochastic Differential Equations i.e. BSDEs are well posed under weak
assumptions on the terminal condition (see  \citeauthor*{pardoux1990adapted} \cite{pardoux1990adapted},
\citeauthor*{el1997backward} \cite{el1997backward}, $\dots$). Indeed, by the smoothing effect of the
heat kernel, or equivalently by an integration--by--parts formula, one does not
need differentiability of the terminal map $x\mapsto u(T,x)$ in order to
obtain differentiability of the propagated map for $t<T$. Moreover, in finite
dimension, \eqref{eq:martingale} and \eqref{eq:martingale_drift} are essentially
equivalent through Girsanov's theorem. Hence one may focus on the driftless case
without losing the substance of the argument.

\medskip
\paragraph*{Our contribution.}The purpose of this paper is to formulate and prove an analogue of the preceding
properties for a suitable semigroup associated with McKean--Vlasov dynamics in the presence of common noise. 
A main difficulty in the space of probability measures is that these convenient
features no longer transfer directly. In particular, the analogue of the
Girsanov reduction is not available in the same way, and the smoothing mechanism
of a possible heat semigroup on the Wasserstein space must be analyzed directly.

\medskip
Let
\[
\Gc_t:=\sigma\{\sigma_\circ W^\circ_r:\ 0\le r\le t\}
\]
be the filtration generated by the common noise. If
\[
X_t=X_0+W_t+\sigma_\circ W^\circ_t,
\]
then the relevant state variable is the conditional law
\[
\mu_t:=\Lc(X_t\mid\Gc_t).
\]
Let $\Uc:[0,T]\times\Pc(\R^d)\to\R$ be a smooth functional. The Itô formula
on the flow $(\mu_t)_{0\le t\le T}$ (see, e.g., \cite[Theorem 4.14]{carmona2018probabilisticII}) gives, at least formally,
\[
\begin{aligned}
\Uc(t,\mu_t)
&=
\Uc(T,\mu_T)
-
\int_t^T
\left(
\int_{\R^d}
\nabla_x\delta_m\Uc(s,\mu_s)(x)\,\mu_s(\mathrm{d}x)
\right)
\cdot\sigma_\circ\,\mathrm{d}W^\circ_s
\\
&\quad
-\int_t^T
\Bigg[
\partial_t\Uc(s,\mu_s)
+
\frac12
\int_{\R^d}
\trace \left[
\nabla_x^2\delta_m\Uc(s,\mu_s)(x)
\bigl(\Ir_d+\sigma_\circ\sigma_\circ^\top\bigr)
\right]\mu_s(\mathrm{d}x)
\\
&\qquad\qquad
+
\frac12
\int_{\R^d\times\R^d}
\trace \left[
\nabla_y\nabla_x\delta_m^2\Uc(s,\mu_s)(x,y)
\sigma_\circ\sigma_\circ^\top
\right]\mu_s(\mathrm{d}x)\mu_s(\mathrm{d}y)
\Bigg]\mathrm{d}s,
\end{aligned}
\]
where the notation $\delta_m \Uc$ stands for the linear functional derivative of $\Uc$ (see equation \eqref{eq:def_linear_func} below for the definition).
The last term is specific to the common-noise setting and comes from the
quadratic variation of the conditional law itself. Thus the natural analogue of
the heat equation on the Wasserstein space is
\[
\begin{aligned}
0
&=
\partial_t\Uc(t,m)
+
\frac12
\int_{\R^d}
\trace \left[
\nabla_x^2\delta_m\Uc(t,m)(x)
\bigl(\Ir_d+\sigma_\circ\sigma_\circ^\top\bigr)
\right]m(\mathrm{d}x)
\\
&\quad+
\frac12
\int_{\R^d\times\R^d}
\trace \left[
\nabla_y\nabla_x\delta_m^2\Uc(t,m)(x,y)
\sigma_\circ\sigma_\circ^\top
\right]m(\mathrm{d}x)m(\mathrm{d}y).
\end{aligned}
\]
If $\Uc$ solves this equation, then the finite-variation part vanishes and
one obtains
\[
\Uc(t,\mu_t)
=
\Uc(T,\mu_T)
-
\int_t^T
\left(
\int_{\R^d}
\nabla_x\delta_m\Uc(s,\mu_s)(x)\,\mu_s(\mathrm{d}x)
\right)
\cdot\sigma_\circ\,\mathrm{d}W^\circ_s .
\]
As in the finite-dimensional case, the backward dynamics then involves only
the first-order derivative $\nabla_x\delta_m\Uc$ which is usually called the Lions derivative. The time derivative
and the second--order derivatives in the measure variable are no longer visible
in the representation.

\medskip

This leads to the analog of the converse question raised in \eqref{eq:martingale}. What are the
properties of a pair $(\Uc,\Vc)$ satisfying
\[
\Uc(t,\mu_t)
=
\Uc(T,\mu_T)
-
\int_t^T
\left(
\int_{\R^d}
\Vc(s,\mu_s)(x)\,\mu_s(\mathrm{d}x)
\right)
\cdot\sigma_\circ\,\mathrm{d}W^\circ_s ?
\]
Can one construct an analog of the heat semigroup on the Wasserstein space
which regularizes a terminal functional and produces such a representation
using only first-order derivatives? As in the finite-dimensional case, such a representation is expected to be
closely related to parabolic equations on the space of probability measures and
to their probabilistic interpretation. In particular, it naturally connects with
the analysis of master equations and with the probabilistic approach to mean
field control and mean field games; see, for instance,
\citeauthor*{cardaliaguet2013notes}~\cite{cardaliaguet2013notes} and
\citeauthor*{carmona2018probabilisticI}~\cite{carmona2018probabilisticI,carmona2018probabilisticII}.

\medskip
There is, however, an important distinction with the finite-dimensional case.
The preceding formulation only concerns the particular flow
$(\mu_t)_{0\le t\le T}$, namely the conditional law of a Brownian diffusion
with common noise. This is the analog of the driftless identity
\eqref{eq:martingale}. To obtain the counterpart of
\eqref{eq:martingale_drift}, one must understand the behavior of the potential
semigroup along the conditional law of a general diffusion,
\[
\mathrm{d}X_t
=
\beta(t,X_t,\nu_t)\,\mathrm{d}t
+
\mathrm{d}W_t+\sigma_\circ\,\mathrm{d}W^\circ_t,
\qquad
\nu_t:=\Lc(X_t\mid\Gc_t).
\]
In finite dimension, this passage is handled by Girsanov's theorem. In the
measure-valued setting, such a reduction does not capture the difficulty:
the state variable is the conditional law itself, and changing the drift changes
the full random flow of conditional distributions. The heat-semigroup effect
must therefore be studied directly along general admissible conditional law
flows.

\medskip
A natural candidate is obtained by propagating the initial law through a
conditional McKean--Vlasov flow. Let $B$ be a reference drift and let
$S^{t,U}$ solve
\[
S^{t,U}_r=U,\qquad r\in[0,t],
\]
and, for $r\in[t,T]$,
\[
\mathrm{d}S^{t,U}_r
=
B(r,S^{t,U}_r,\mu^{t,m}_r)\,\mathrm{d}r
+\mathrm{d}W_r+\sigma_\circ\,\mathrm{d}W^\circ_r,
\qquad
\mu^{t,m}_r:=\Lc(S^{t,U}_r\mid\Gc_r),
\]
where $U\sim m$ is independent of the noises. For a terminal functional
$G:\Pc(\R^d)\to\R$, define
\[
\Kc_BG(t,m)
:=
\E\bigl[G(\mu^{t,m}_T)\bigr].
\]

The dynamics considered here are closely related to the literature on nonlinear
Markov processes and McKean--Vlasov semigroups. The terminology goes back to
McKean's interpretation of nonlinear evolution equations through stochastic
processes whose coefficients depend on their own law. In this setting, the
flow $m\longmapsto \mu^{t,m}_T$
is often called the McKean--Vlasov semigroup or distribution semigroup. General
perspectives on nonlinear Markov semigroups can be found in
\citeauthor{kolokoltsov2006nonlinear} \cite{kolokoltsov2006nonlinear}, while second--order differential
properties of McKean--Vlasov semigroups are studied in
\citeauthor*{arnaudon2020second} \cite{arnaudon2020second}. See also \citeauthor*{buckdahn2014mean} \cite{buckdahn2014mean}, \citeauthor*{Frikha2022Well} \cite{Frikha2022Well}, $\cdots$ for the interpretation as solution of PDEs on the Wasserstein space. These questions are also
closely connected with propagation of chaos and the approximation of nonlinear
McKean--Vlasov flows by weakly interacting particle systems, see for instance
\citeauthor*{sznitman1991topics} \cite{sznitman1991topics}, \citeauthor*{meleard1996asymptotic} \cite{meleard1996asymptotic}.

\medskip
The present work adopts a complementary perspective, which consists of considering the semi--group as an operator taking a map as an input (as in the finite-dimensional case). Instead of focusing primarily
on stability estimates or higher--order differentiability properties of the
flow, we study its martingale representation property. More
precisely, we ask whether the semigroup generated by a reference
McKean--Vlasov dynamics can play, on the Wasserstein space, the same role as the
semigroup in finite dimension. In this sense, we shall view $\Kc_B$ as the McKean--Vlasov semigroup, or Wasserstein heat semigroup, associated with
the reference drift $B$. The main objective is to show that this operator
produces a first-order backward representation along general admissible
conditional law flows, without assuming a priori that this semigroup has a time derivative or second--order derivatives in the measure variable, in direct analogy with the heat--kernel representation in
the classical Markovian setting.
More precisely, let
$(\nu_t)_{0\le t\le T}$ be generated by a possibly different drift $\beta$:
\[
\mathrm{d}X_t
=
\beta(t,X_t,\nu_t)\,\mathrm{d}t
+
\mathrm{d}W_t+\sigma_\circ\,\mathrm{d}W^\circ_t,
\qquad
\nu_t:=\Lc(X_t\mid\Gc_t).
\]
Then, under suitable assumptions,
\[
\begin{aligned}
\Kc_BG(t,\nu_t)
&=
G(\nu_T)
+
\int_t^T
\E\Big[
\nabla_x\delta_m\Kc_BG(r,\nu_r)(X_r)
\cdot
\bigl(B(r,X_r,\nu_r)-\beta(r,X_r,\nu_r)\bigr)
\,\Big|\,\Gc_r
\Big]\mathrm{d}r
\\
&\quad
-
\int_t^T
\E\Big[
\nabla_x\delta_m\Kc_BG(r,\nu_r)(X_r)
\,\Big|\,\Gc_r
\Big]\cdot\sigma_\circ\,\mathrm{d}W^\circ_r.
\end{aligned}
\]
When $B=0$, this becomes
\[
\begin{aligned}
\Kc_0G(t,\nu_t)
=
G(\nu_T)
&-
\int_t^T
\E\Big[
\nabla_x\delta_m\Kc_0G(r,\nu_r)(X_r)
\cdot
\beta(r,X_r,\nu_r)
\,\Big|\,\Gc_r
\Big]\mathrm{d}r
\\
&-
\int_t^T
\E\Big[
\nabla_x\delta_m\Kc_0G(r,\nu_r)(X_r)
\,\Big|\,\Gc_r
\Big]\cdot\sigma_\circ\,\mathrm{d}W^\circ_r.
\end{aligned}
\]
This formula may be interpreted as a martingale representation theorem on the
space of probability measures. The conditional law is the state variable, the
common noise generates the martingale part, and the Lions derivative
$\nabla_x\delta_m\Kc_BG$ plays the role of the classical gradient.

\medskip
The comparison with the finite-dimensional formula is direct. The gradient
$\nabla_xu$ is replaced by $\nabla_x\delta_m\Kc_BG$. The Brownian martingale
term is replaced by a martingale driven by the common noise. The drift correction
$B-\beta$ measures the discrepancy between the reference McKean--Vlasov
dynamics used to define the semigroup and the admissible dynamics along which
the semigroup is evaluated.

\smallskip
There is nevertheless an important difference. In finite dimension, the heat
kernel can regularize very rough terminal functions. In the Wasserstein setting,
one cannot expect such a result without any regularity on
\[
m\longmapsto G(m).
\]
Some first-order differentiability of $G$ is needed. The regularization
effect remains substantial, however: the representation requires neither a time
derivative of the semigroup nor a second--order derivative in the measure
variable. Moreover, when $B=0$, the assumptions can be weakened further: the
Brownian conditional flow allows one to work with the linear functional
derivative $\delta_mG$, without imposing the terminal Lions derivative
$\nabla_x\delta_mG$.

\medskip
A further consequence of this dynamic viewpoint is that the framework can accommodate a discrete form of path dependence. More precisely, we consider terminal functionals depending on the conditional law at finitely many deterministic times, 
\[ G\bigl(\mu_{t_0},\ldots,\mu_{t_k}\bigr), \qquad 0=t_0<t_1<\cdots<t_k=T. 
\] 
Although such functionals are no longer Markovian in the current measure alone, the preceding representation can be applied recursively on each interval $[t_j,t_{j+1}]$. At every step, the previously observed values $\mu_{t_0},\ldots,\mu_{t_j}$ are frozen as parameters, while the current conditional law remains the Markovian state variable. A backward induction over the grid then yields a martingale representation for the resulting discrete path-dependent functional. Thus, discrete path dependence can be incorporated without developing a new representation theorem on the full space of measure-valued paths.

\medskip
\paragraph*{Strategy of the proof.}Our proof follows a route which is somewhat different from the usual approach
to McKean--Vlasov semigroups. Rather than deriving and using a differential
equation on the space of probability measures, which would require one time
derivative and second--order derivatives in the measure variable, we work
directly at the level of a finite-particle approximation. This allows us to
identify the first-order backward dynamics of the semigroup without ever
differentiating twice in the measure argument. Another important feature of the
argument is that it naturally incorporates common noise, a situation which is
less often treated in the existing literature on McKean--Vlasov semigroups. In
this sense, our approach is close in spirit to the particle-based analysis of
McKean--Vlasov semigroups, but it is adapted here to the conditional-law setting
induced by the common noise.
More precisely, at the particle level, we consider the map
\[
\xbb=(x_1,\dots,x_n)
\longmapsto
G\left(\frac1n\sum_{i=1}^n\delta_{X^{i,n,t,\xbb}_T}\right),
\]
where $(X^{1,n,t,\xbb},\dots,X^{n,n,t,\xbb})$ is the interacting particle
system starting from $\xbb$ at time $t$. Differentiating this finite-dimensional
map with respect to one coordinate $x_i$ produces two distinct contributions.

\smallskip
The first one is the diagonal contribution. It corresponds to the sensitivity of
particle $i$ with respect to its own initial condition $x_i$. This term is
of order one and converges to the usual tangent process along the limiting
McKean--Vlasov dynamics.
The second one is the off-diagonal contribution. It corresponds to the
sensitivity of a particle $j$, with $j\neq i$, with respect to the initial
condition $x_i$. This effect is transmitted only through the empirical
measure, and each individual off-diagonal derivative is therefore of order
$1/n$. However, since there are $n(n-1)$ such interactions, their cumulative
contribution is non-negligible and survives in the mean-field limit. The limit
of this contribution is precisely the additional term appearing in the Lions
derivative of the propagated functional.

\smallskip
This diagonal/off-diagonal decomposition is the main structural feature of the
proof. It explains why the Wasserstein semigroup differs from the classical
finite-dimensional heat semigroup: perturbing one initial particle not only
moves this particle, but also perturbs the empirical distribution, and this
perturbation propagates through the whole system. The finite-particle
approximation makes this mechanism explicit and provides a direct way to obtain
the desired first-order representation in the presence of common noise. The techniques developed in this paper are inspired by, and share several features with, those introduced in \citeauthor*{djete2025notionbsdewassersteinspace}~\cite{djete2025notionbsdewassersteinspace}.

\medskip
\paragraph*{Organization of the paper.} In \Cref{sec:notation}, we introduce the probabilistic framework, the notion of linear functional derivative, and the standing assumptions. In \Cref{sec:integral_formula}, we study the conditional McKean--Vlasov flow and its diagonal and off-diagonal tangent processes. Using a finite-particle approximation, we derive an integral formula for the variation of the propagated functional and identify its Lions derivative. In \Cref{sec:backward_flow}, we prove the main backward flow formula for the McKean--Vlasov semigroup along arbitrary admissible conditional law flows. We also discuss its interpretation as a nonlinear Markov semigroup and as a Feynman--Kac representation for linear parabolic equations on the space of probability measures. In \Cref{sec:martingale_representation}, we specialize to the driftless reference case and formulate an intrinsic martingale representation problem, whose unique solution is given by the McKean--Vlasov heat semigroup and its Lions derivative. Finally, \Cref{sec:mixed_formula} develops a mixed finite-dimensional/Wasserstein backward formula, while \Cref{sec:path_dependent} shows how discrete path dependence can be incorporated by iterating the Markovian representation backward along a deterministic time grid.

\section{Notation and assumptions} \label{sec:notation}

Let $T>0$, $d\ge1$, and let $(\Omega,\Fc,\F,\P)$ be a filtered probability
space supporting two $\R^d$--valued independent Brownian motions:
\[
W=(W_t)_{0\le t\le T},
\qquad
W^\circ=(W^\circ_t)_{0\le t\le T}.
\]
The process $W$ is the idiosyncratic noise and $W^\circ$ is the common noise.
We fix a matrix $\sigma_\circ\in\R^{d\times d}$ and set
\[
\Gc_t:=\sigma\{\sigma_\circ W^\circ_r:\ 0\le r\le t\}.
\]
We write $\Pc(\R^d)$ for the set of probability measures on $\R^d$ equipped with the weak convergence topology.

\medskip
We use the following convention for derivatives in the measure variable. A map
$F:\Pc(\R^d)\to\R$ admits a linear functional derivative if there exists a Borel map
$\delta_mF:\Pc(\R^d)\times\R^d\to\R$ such that, for every
$m,m'\in\Pc(\R^d)$ such that $\int_0^1
\int_{\R^d}
|\delta_mF\bigl((1-\theta)m+\theta m'\bigr)(x)|
\,(m'+m)(\mathrm{d}x)\,\mathrm{d}\theta< \infty$, we have
\begin{align} \label{eq:def_linear_func}
F(m')-F(m)
=
\int_0^1
\int_{\R^d}
\delta_mF\bigl((1-\theta)m+\theta m'\bigr)(x)
\,(m'-m)(\mathrm{d}x)\,\mathrm{d}\theta .
\end{align}
The derivative $\delta_mF$ is defined up to an additive constant. 
Whenever $x\mapsto\delta_mF(m)(x)$ is differentiable, we denote by
$\nabla_x\delta_mF(m)(x)$ its spatial derivative. This is the Lions derivative
of $F$, up to the usual identification. When $m \longmapsto \delta_mF(m)(x)$ itself admits a linear functional derivative,
we denote it by $\delta^2_mF(m)(x,y)$, and we write $\nabla_y\nabla_x\delta^2_mF(m)(x,y)$
for its mixed spatial gradient whenever it exists.

For a map
\[
B:[0,T]\times\R^d\times\Pc(\R^d)\to\R^d,
\]
we write
\[
\nabla_xB(t,x,m)\in\R^{d\times d},
\]
and
\[
\delta_mB(t,x,m)(y)\in\R^d,
\qquad
\nabla_y\delta_mB(t,x,m)(y)\in\R^{d\times d}.
\]

\begin{assumption}
\label{assum:main}
The drift $B$ is bounded and continuous in $(x,m)$ for each $t$. It is differentiable in the space
variable and admits a linear functional derivative in the measure variable.
Moreover,
\[
\nabla_xB,
\qquad
\delta_mB,
\qquad
\nabla_y\delta_mB
\]
are bounded and continuous in $(x,m,y)$ for each $t$.
\end{assumption}

\begin{assumption}
\label{assum:G}
The terminal functional $G:\Pc(\R^d)\to\R$ is bounded and continuous. It
admits a linear functional derivative whose spatial gradient
\[
\nabla_x\delta_mG(m)(x)
\]
is bounded and continuous.
\end{assumption}

The assumptions above are deliberately stated in a smooth regime. They are not
intended to be optimal. Their role is to justify the differentiations and the
passage from finite-particle systems to conditional McKean--Vlasov limits. One
of the motivations of the paper is precisely that, once the semigroup is
constructed, the backward formula only uses first-order derivatives.

\section{Integral formula along the McKean--Vlasov flow} \label{sec:integral_formula}

\paragraph*{The conditional McKean--Vlasov flow}

Let $0\le t\le s\le T$ and $m\in\Pc(\R^d)$. Let $U$ be an
$\R^d$--valued random variable with law $m$, independent of $W$ and
$W^\circ$. We denote by $S^{t,m}$ the solution of
\[
S^{t,m}_r=U,\qquad r\in[0,t],
\]
and, for $r\in[t,T]$,
\[
\mathrm{d}S^{t,m}_r
=
B(r,S^{t,m}_r,\mu^{t,m}_r)\,\mathrm{d}r
+\mathrm{d}W_r+\sigma_\circ\,\mathrm{d}W^\circ_r,\qquad \mu^{t,m}_r:=\Lc(S^{t,m}_r\mid\Gc_r).
\]
Since $S^{t,m}_r$ only depends on the common noise up to time $r$, one also
has
\[
\Lc(S^{t,m}_r\mid\Gc_T)=\Lc(S^{t,m}_r\mid\Gc_r).
\]
We shall freely use either conditioning. Under our standing assumptions, the well--posedness (i.e., existence and uniqueness) of the conditional
McKean--Vlasov equation defining $S^{t,m}$ follows from \citeauthor*{djete2019mckean} 
\cite[Theorem A.3]{djete2019mckean}.
For a terminal functional $G$, define
\[
\Gr(s,t,m):=G(\mu^{t,m}_s),
\]
which is in general $\Gc_s$--measurable. We also define its averaged version
\[
\overline\Gr(s,t,m):=\E\bigl[G(\mu^{t,m}_s)\bigr].
\]

The object $\overline\Gr$ is deterministic. In particular, the McKean--Vlasov
semigroup associated with the terminal time $T$ is
\[
\Kc_BG(t,m):=\overline\Gr(T,t,m)=\E\bigl[G(\mu^{t,m}_T)\bigr].
\]

\paragraph*{Tangent processes}

The derivative of the McKean--Vlasov flow with respect to the initial law has
two components. The first component is diagonal and corresponds to the
variation of a particle with respect to its own initial condition. The second is
off-diagonal and corresponds to the limiting effect of perturbing one particle
on another particle through the empirical measure.
Let $m,m'\in\Pc(\R^d)$, and let $(U,U')$ be a coupling of $(m,m')$,
independent of the noises. For $\theta\in[0,1]$, define
\[
U^\theta:=U'+\theta(U-U'),
\qquad
\eta^\theta:=\Lc(U^\theta).
\]
Let $S^\theta=S^{t,\eta^\theta}$, and write
\[
\mu^\theta_r:=\Lc(S^\theta_r\mid\Gc_r).
\]

The diagonal tangent process $J^\theta$ is defined by
\[
J^\theta_r=\Ir_d,\qquad r\in[0,t],
\]
and, for $r\in[t,T]$,
\[
\mathrm{d}J^\theta_r
=
\nabla_xB(r,S^\theta_r,\mu^\theta_r)J^\theta_r\,\mathrm{d}r.
\]
With the well--posedness of $S^\theta$, the well--posedness of the couple $(S^\theta, J^\theta)$ follows directly.
Let $\overline S^\theta$ be a conditionally independent copy of $S^\theta$
given the common noise. Thus
\[
\overline S^\theta_r=\overline U^\theta,\qquad r\in[0,t],
\]
and
\[
\mathrm{d}\overline S^\theta_r
=
B(r,\overline S^\theta_r,\mu^\theta_r)\,\mathrm{d}r
+\mathrm{d}\overline W_r+\sigma_\circ\,\mathrm{d}W^\circ_r.
\]
The off-diagonal tangent process $\overline J^\theta$ starts from zero:
\[
\overline J^\theta_r=0,\qquad r\in[0,t].
\]
For $r\in[t,T]$, it solves
\begin{align} \label{eq:off_diagonal}
\mathrm{d}\overline J^\theta_r
&=
\nabla_xB(r,\overline S^\theta_r,\mu^\theta_r)
\overline J^\theta_r\,\mathrm{d}r
+
\nabla_y\delta_mB(r,\overline S^\theta_r,\mu^\theta_r)(S^\theta_r)
J^\theta_r\,\mathrm{d}r
+
\int_{\R^d\times\R^{d\times d}}
\nabla_y\delta_mB(r,\overline S^\theta_r,\mu^\theta_r)(x)j\,
\widehat\mu^\theta_r(\mathrm{d}x,\mathrm{d}j)\,\mathrm{d}r,
\end{align}
where
\[
\widehat\mu^\theta_r
:=
\Lc(\overline S^\theta_r,\overline J^\theta_r\mid S^\theta,J^\theta,\Gc_r).
\]

The second term represents the direct effect of the tagged perturbation
$(S^\theta,J^\theta)$ on the copy $\overline S^\theta$. The third term
represents the average effect of the perturbation after it has propagated
through the population.
 Using the well--posedness of $\left(S^\theta,J^\theta,\overline S^\theta \right)$, another application of \cite[Theorem A.3]{djete2019mckean} provides the well--posedness of $\left(S^\theta,J^\theta,\overline S^\theta, \overline J^\theta \right)$.

\begin{proposition}[Integral formula along a McKean--Vlasov flow]
\label{prop:integral_formula_flow}
Assume {\rm \Cref{assum:main,assum:G}}. Let $m,m'\in\Pc(\R^d)$, and let
$(U,U')$ be a coupling of $(m,m')$, independent of the noises. Then, for
every $s\in[t,T]$, $\P$-a.s.,
\[
\begin{aligned}
G(\mu^{t,m}_s)-G(\mu^{t,m'}_s)
&=
\int_0^1
\E\Big[
\nabla_x\delta_mG(\mu^\theta_s)(S^\theta_s)
\cdot J^\theta_s(U-U')
\,\Big|\,\Gc_s
\Big]\,\mathrm{d}\theta
+
\int_0^1
\E\Big[
\nabla_x\delta_mG(\mu^\theta_s)(\overline S^\theta_s)
\cdot \overline J^\theta_s(U-U')
\,\Big|\,\Gc_s
\Big]\,\mathrm{d}\theta.
\end{aligned}
\]
\end{proposition}

\begin{proof}
The proof is based on a finite-particle approximation. The purpose of the
argument is to identify the derivative of
\[
m\longmapsto G(\mu^{t,m}_s)
\]
without differentiating a nonlinear stochastic flow directly on
$\Pc(\R^d)$. We approximate the flow by an interacting particle system,
differentiate the finite-dimensional map, and then pass to the limit.

\medskip

\noindent
\textbf{Step 1. finite-particle approximation.}

Let
\[
\xi=(x_1,\dots,x_n)\in(\R^d)^n.
\]
We consider the $n$-particle system
\[
X^{i,n,\xi}_r=x_i,\qquad r\in[0,t],
\]
and, for $r\in[t,T]$,
\[
\mathrm{d}X^{i,n,\xi}_r
=
B(r,X^{i,n,\xi}_r,\mu^{n,\xi}_r)\,\mathrm{d}r
+\mathrm{d}W^i_r+\sigma_\circ\,\mathrm{d}W^\circ_r,
\]
where
\[
\mu^{n,\xi}_r:=\frac1n\sum_{k=1}^n\delta_{X^{k,n,\xi}_r}.
\]
For fixed $s\in[t,T]$, set
\[
G^n_s(\xi):=G(\mu^{n,\xi}_s).
\]

Let
\[
\xi'=(x'_1,\dots,x'_n),
\qquad
\xi^\theta:=\xi'+\theta(\xi-\xi'),
\qquad \theta\in[0,1].
\]
Since $G^n_s$ is a smooth map on $(\R^d)^n$, the fundamental theorem of
calculus gives
\[
G^n_s(\xi)-G^n_s(\xi')
=
\int_0^1
\sum_{i=1}^n
D_{x_i}G^n_s(\xi^\theta)\cdot(x_i-x'_i)\,\mathrm{d}\theta.
\]
The problem is therefore reduced to computing $D_{x_i}G^n_s$.

\medskip

\noindent
\textbf{Step 2. Differentiating the particle system.}

For $1\le i,j\le n$, define
\[
J^{j,i,n,\xi}_r:=D_{x_i}X^{j,n,\xi}_r.
\]
This is a $d\times d$-valued process. It measures how particle $j$ reacts to
a perturbation of the initial condition of particle $i$. For $r\le t$,
\[
J^{j,i,n,\xi}_r=\1_{\{i=j\}}\Ir_d.
\]
For $r\in[t,T]$, differentiating the particle dynamics yields
\[
\begin{aligned}
\mathrm{d}J^{j,i,n,\xi}_r
&=
\nabla_xB(r,X^{j,n,\xi}_r,\mu^{n,\xi}_r)
J^{j,i,n,\xi}_r\,\mathrm{d}r
\\
&\quad+
\frac1n\sum_{k=1}^n
\nabla_y\delta_mB(r,X^{j,n,\xi}_r,\mu^{n,\xi}_r)(X^{k,n,\xi}_r)
J^{k,i,n,\xi}_r\,\mathrm{d}r.
\end{aligned}
\]
The factor $1/n$ comes from the empirical measure. It is the source of the
off-diagonal scaling below.

By the chain rule for functions of empirical measures,
\[
D_{x_i}G^n_s(\xi)
=
\frac1n\sum_{j=1}^n
\bigl(J^{j,i,n,\xi}_s\bigr)^\top
\nabla_x\delta_mG(\mu^{n,\xi}_s)(X^{j,n,\xi}_s).
\]
Thus, for any $h_i\in\R^d$,
\[
D_{x_i}G^n_s(\xi)\cdot h_i
=
\frac1n\sum_{j=1}^n
\nabla_x\delta_mG(\mu^{n,\xi}_s)(X^{j,n,\xi}_s)
\cdot J^{j,i,n,\xi}_s h_i.
\]

We now separate the term $j=i$ from the terms $j\neq i$. Applying the
previous identity with $h_i=x_i-x'_i$, we obtain
\[
\begin{aligned}
G^n_s(\xi)-G^n_s(\xi')
&=
\int_0^1
\frac1n\sum_{i=1}^n
\nabla_x\delta_mG(\mu^{n,\xi^\theta}_s)(X^{i,n,\xi^\theta}_s)
\cdot
J^{i,i,n,\xi^\theta}_s(x_i-x'_i)
\,\mathrm{d}\theta
\\
&\quad+
\int_0^1
\frac1{n^2}\sum_{i=1}^n\sum_{j\neq i}
\nabla_x\delta_mG(\mu^{n,\xi^\theta}_s)(X^{j,n,\xi^\theta}_s)
\cdot
\bigl(nJ^{j,i,n,\xi^\theta}_s(x_i-x'_i)\bigr)
\,\mathrm{d}\theta.
\end{aligned}
\]
This is the key decomposition. The first line is the diagonal contribution.
The second line is the off-diagonal contribution. For $j\neq i$, the process
$J^{j,i,n}$ is of order $1/n$, but there are $n(n-1)$ such terms. Hence
the rescaled quantity $nJ^{j,i,n}$ has a non-trivial limit.

\medskip

\noindent
\textbf{Step 3. Randomizing the initial condition.}

Let $(U^i,U^{\prime i})_{i\ge1}$ be i.i.d. copies of $(U,U')$, independent
of all Brownian motions. Define
\[
U^{\theta,i}:=U^{\prime i}+\theta(U^i-U^{\prime i}),
\qquad
\Ubb^{\theta,n}:=(U^{\theta,1},\dots,U^{\theta,n}).
\]
We apply the previous identity with
\[
\xi=(U^1,\dots,U^n),
\qquad
\xi'=(U^{\prime1},\dots,U^{\prime n}).
\]
Conditional propagation of chaos gives, for each $r\in[t,T]$, 
\[
\mu^{n,\Ubb^{\theta,n}}_r
\Longrightarrow
\mu^\theta_r
=
\Lc(S^\theta_r\mid\Gc_r),
\]
where the convergence holds in the weak sense and holds in probability, conditionally on the common noise.

\medskip

\noindent
\textbf{Step 4. Uniform estimates for tangent processes.}

The tangent processes satisfy the estimate
\[
\sup_{n\ge1}\sup_{1\le i\le n}
\frac1n\sum_{j=1}^n
\E\left[
\sup_{r\in[t,T]}
|J^{j,i,n,\Ubb^{\theta,n}}_r|^p
\left(\1_{\{j=i\}}+n^p\1_{\{j\neq i\}}\right)
\right]
<\infty,
\]
for every $p\ge1$.

Indeed, $J^{i,i,n}$ starts from $\Ir_d$, and therefore remains of order one by
Gronwall's lemma. On the other hand, $J^{j,i,n}$, $j\neq i$, starts from
zero and is created only through the empirical-measure derivative, which carries
a factor $1/n$. Hence $J^{j,i,n}$ is of order $1/n$.

A useful consequence is
\[
\lim_{n\to\infty}
\E\left[
\int_t^T
\left|
\frac1n\sum_{k=1}^n
\nabla_y\delta_mB(r,X^{i,n}_r,\mu^n_r)(X^{k,n}_r)
J^{k,i,n}_r
\right|
\,\mathrm{d}r
\right]
=0.
\]
The term $k=i$ is multiplied by $1/n$, whereas the terms $k\neq i$ contain
$J^{k,i,n}=O(1/n)$.

\medskip

\noindent
\textbf{Step 5. Limit of the diagonal contribution.}

The diagonal tangent process satisfies
\[
\begin{aligned}
\mathrm{d}J^{i,i,n}_r
&=
\nabla_xB(r,X^{i,n}_r,\mu^n_r)J^{i,i,n}_r\,\mathrm{d}r
\\
&\quad+
\frac1n\sum_{k=1}^n
\nabla_y\delta_mB(r,X^{i,n}_r,\mu^n_r)(X^{k,n}_r)
J^{k,i,n}_r\,\mathrm{d}r.
\end{aligned}
\]
By the previous estimate, the second line vanishes in the limit. Therefore, by conditional propagation of chaos, see for instance \cite[Proposition 4.15]{djete2019general}, we obtain, for the weak topology, 
\[ 
    \frac1n\sum_{i=1}^n \delta_{\left(J^{i,i,n},\,X^{i,n},\,U^i,\,U^{\prime i},\,\mu^n\right)} \;\xrightarrow[n\to\infty]{}\; \Lc\left(J^\theta,S^\theta,U,U',\mu^\theta\,\middle|\,\Gc_T\right), \qquad \text{a.s.}
\]
Here $J^\theta$ is the diagonal tangent process associated with the limiting conditional McKean--Vlasov flow. It is initialized by \[ J^\theta_r=\Ir_d,\qquad r\in[0,t], \] and solves, for $r\in[t,T]$, 
\[ 
    \mathrm{d}J^\theta_r = \nabla_xB(r,S^\theta_r,\mu^\theta_r)J^\theta_r\,\mathrm{d}r .
\]
Consequently, for every bounded $\Gc_s$-measurable random variable $H$,
\[
\begin{aligned}
&\lim_{n\to\infty}
\E\bigg[
H
\int_0^1
\frac1n\sum_{i=1}^n
\nabla_x\delta_mG(\mu^{n,\Ubb^{\theta,n}}_s)(X^{i,n,\Ubb^{\theta,n}}_s)
\cdot
J^{i,i,n,\Ubb^{\theta,n}}_s(U^i-U^{\prime i})
\,\mathrm{d}\theta
\bigg]
\\
&\qquad =
\E\bigg[
H
\int_0^1
\E\Big[
\nabla_x\delta_mG(\mu^\theta_s)(S^\theta_s)
\cdot
J^\theta_s(U-U')
\,\Big|\,\Gc_s
\Big]\mathrm{d}\theta
\bigg].
\end{aligned}
\]

\medskip

\noindent
\textbf{Step 6. Limit of the off-diagonal contribution.}

For $j\neq i$, set
\[
\Gamma^{j,i,n}_r:=nJ^{j,i,n}_r.
\]
Then $\Gamma^{j,i,n}_r=0$ for $r\le t$, and
\[
\begin{aligned}
\mathrm{d}\Gamma^{j,i,n}_r
&=
\nabla_xB(r,X^{j,n}_r,\mu^n_r)\Gamma^{j,i,n}_r\,\mathrm{d}r
\\
&\quad+
\nabla_y\delta_mB(r,X^{j,n}_r,\mu^n_r)(X^{i,n}_r)
J^{i,i,n}_r\,\mathrm{d}r
\\
&\quad+
\frac1n\sum_{k\neq i}
\nabla_y\delta_mB(r,X^{j,n}_r,\mu^n_r)(X^{k,n}_r)
\Gamma^{k,i,n}_r\,\mathrm{d}r.
\end{aligned}
\]
The first line is the transport of the perturbation along particle $j$. The
second line is the direct effect of the tagged particle $i$. The third line is
the average propagation through the rest of the population.
The estimates of Step~4 imply the tightness of the sequence of laws
$(\Q^n)_{n\ge 1}$, where
\[
\begin{aligned}
\Q^n
&:=
\frac1n\sum_{i=1}^n
\Lc\left(
X^{i,n},J^{i,i,n},U^i,U^{\prime i},
\widehat\mu^{i,n},\mu^n,W^\circ
\right)
\\
&=
\Lc\left(
X^{1,n},J^{1,1,n},U^1,U^{\prime 1},
\widehat\mu^{1,n},\mu^n,W^\circ
\right),
\end{aligned}
\]
with
\[
\widehat\mu^{i,n}
:=
\frac1{n-1}\sum_{j\neq i}
\delta_{\left(X^{j,n},\Gamma^{j,i,n}\right)} .
\]
Here $\Gamma^{j,i,n}=nJ^{j,i,n}$ denotes the rescaled off-diagonal tangent
process.

Using the propagation of chaos established in Step~5, together with the
uniform estimates on the rescaled off-diagonal tangent processes, one identifies
any limit point of $(\Q^n)_{n\ge1}$ as
\[
\Lc\left(
S^\theta,J^\theta,U,U',
\widehat\mu^\theta,\mu^\theta,W^\circ
\right),
\]
where
\[
\widehat\mu^\theta
:=
\Lc\left(
\overline S^\theta,\overline J^\theta
\,\middle|\,
S^\theta,J^\theta,U,U',\Gc_T
\right).
\]
Equivalently,
\[
\begin{aligned}
\Lc\left(
X^{1,n},J^{1,1,n},U^1,U^{\prime 1},
\widehat\mu^{1,n},\mu^n,W^\circ
\right)
\Longrightarrow
\Lc\left(
S^\theta,J^\theta,U,U',
\widehat\mu^\theta,\mu^\theta,W^\circ
\right).
\end{aligned}
\]
In this limiting system, $\overline S^\theta$ is conditionally independent of
$S^\theta$ given the common noise, and $\overline J^\theta$ is the solution
of the off-diagonal linear McKean--Vlasov equation
\eqref{eq:off_diagonal}. Since this linear equation is well posed, the limit
point is unique. Consequently, the whole sequence $(\Q^n)_{n\ge1}$ converges.
Therefore, for every bounded $\Gc_s$-measurable random variable $H$,
\[
\begin{aligned}
&\lim_{n\to\infty}
\E\bigg[
H
\int_0^1
\frac1{n^2}\sum_{i=1}^n\sum_{j\neq i}
\nabla_x\delta_mG(\mu^{n,\Ubb^{\theta,n}}_s)(X^{j,n,\Ubb^{\theta,n}}_s)
\cdot
\Gamma^{j,i,n}_s(U^i-U^{\prime i})
\,\mathrm{d}\theta
\bigg]
\\
&\qquad =
\E\bigg[
H
\int_0^1
\E\Big[
\nabla_x\delta_mG(\mu^\theta_s)(\overline S^\theta_s)
\cdot
\overline J^\theta_s(U-U')
\,\Big|\,\Gc_s
\Big]\mathrm{d}\theta
\bigg].
\end{aligned}
\]

\medskip

\noindent
\textbf{Step 7. Conclusion.}

By conditional propagation of chaos,
\[
G\left(\mu^{n,(U^1,\dots,U^n)}_s\right)
\longrightarrow
G(\mu^{t,m}_s),
\qquad
G\left(\mu^{n,(U^{\prime1},\dots,U^{\prime n})}_s\right)
\longrightarrow
G(\mu^{t,m'}_s),
\]
in probability, and in $L^1$ under the boundedness assumptions.

Combining the diagonal and off-diagonal limits yields, for every bounded
$\Gc_s$-measurable random variable $H$,
\[
\begin{aligned}
\E\Big[H\bigl(G(\mu^{t,m}_s)-G(\mu^{t,m'}_s)\bigr)\Big]
&=
\E\bigg[
H
\int_0^1
\E\Big[
\nabla_x\delta_mG(\mu^\theta_s)(S^\theta_s)
\cdot J^\theta_s(U-U')
\,\Big|\,\Gc_s
\Big]\mathrm{d}\theta
\bigg]
\\
&\quad+
\E\bigg[
H
\int_0^1
\E\Big[
\nabla_x\delta_mG(\mu^\theta_s)(\overline S^\theta_s)
\cdot \overline J^\theta_s(U-U')
\,\Big|\,\Gc_s
\Big]\mathrm{d}\theta
\bigg].
\end{aligned}
\]
Since the identity holds for every bounded $\Gc_s$-measurable $H$, the
conditional identity follows.
\end{proof}

\paragraph*{Lions derivative of the propagated functional}

The previous proposition immediately identifies the Lions derivative of the
propagated functional
\[
m\longmapsto G(\mu^{t,m}_s).
\]
The formula has two terms. The first is the direct derivative along the tagged
particle. The second is the off-diagonal derivative, which accounts for the
variation of the surrounding conditional law.

\begin{corollary}[Derivative of the propagated functional]
\label{cor:derivative_propagated}
Assume {\rm \Cref{assum:main,assum:G}}. Then, for every $0\le t\le s\le T$, the map
\[
m\longmapsto \Gr(s,t,m)=G(\mu^{t,m}_s)
\]
admits a differentiable linear functional derivative. More precisely, for
$m\in\Pc(\R^d)$ and $x\in\R^d$,
\[
\begin{aligned}
\nabla_x\delta_m\Gr(s,t,m)(x)
&=
\E\Big[
\bigl(J^{t,m}_s\bigr)^\top
\nabla_x\delta_mG(\mu^{t,m}_s)(S^{t,m}_s)
\,\Big|\,\Gc_s, U=x
\Big]
\\
&\quad+
\E\Big[
\bigl(\overline J^{t,m}_s\bigr)^\top
\nabla_x\delta_mG(\mu^{t,m}_s)(\overline S^{t,m}_s)
\,\Big|\,\Gc_s, U=x
\Big].
\end{aligned}
\]
Equivalently, for every $h\in\R^d$,
\[
\begin{aligned}
\nabla_x\delta_m\Gr(s,t,m)(x)\cdot h
&=
\E\Big[
\nabla_x\delta_mG(\mu^{t,m}_s)(S^{t,m}_s)
\cdot J^{t,m}_s h
\,\Big|\,\Gc_s,\ U=x
\Big]
\\
&\quad+
\E\Big[
\nabla_x\delta_mG(\mu^{t,m}_s)(\overline S^{t,m}_s)
\cdot \overline J^{t,m}_s h
\,\Big|\,\Gc_s,\ U=x
\Big].
\end{aligned}
\]
\end{corollary}

\begin{proof}
Let $m,m'\in\Pc(\R^d)$, and let $(U,U')$ be a coupling of
$(m,m')$. For $\theta\in[0,1]$, set
\[
U^\theta:=U'+\theta(U-U'),
\qquad
m^\theta:=\Lc(U^\theta).
\]
By \Cref{prop:integral_formula_flow},
\[
\begin{aligned}
\Gr(s,t,m)-\Gr(s,t,m')
&=
\int_0^1
\E\Big[
\nabla_x\delta_mG(\mu^{t,m^\theta}_s)(S^{t,m^\theta}_s)
\cdot
J^{t,m^\theta}_s(U-U')
\,\Big|\,\Gc_s
\Big]\mathrm{d}\theta
\\
&\quad+
\int_0^1
\E\Big[
\nabla_x\delta_mG(\mu^{t,m^\theta}_s)(\overline S^{t,m^\theta}_s)
\cdot
\overline J^{t,m^\theta}_s(U-U')
\,\Big|\,\Gc_s
\Big]\mathrm{d}\theta.
\end{aligned}
\]
This is the fundamental theorem of calculus in the measure argument. To identify
the spatial derivative of the linear functional derivative at $m$, one
considers an infinitesimal perturbation of the initial point $U=x$ in the
direction $h$. The direct perturbation at time $s$ is $J^{t,m}_s h$, while
the perturbation transmitted through the conditional law is
$\overline J^{t,m}_s h$. Hence
\[
\begin{aligned}
\nabla_x\delta_m\Gr(s,t,m)(x)\cdot h
&=
\E\Big[
\nabla_x\delta_mG(\mu^{t,m}_s)(S^{t,m}_s)
\cdot J^{t,m}_s h
\,\Big|\,\Gc_s,\ U=x
\Big]
\\
&\quad+
\E\Big[
\nabla_x\delta_mG(\mu^{t,m}_s)(\overline S^{t,m}_s)
\cdot \overline J^{t,m}_s h
\,\Big|\,\Gc_s,\ U=x
\Big].
\end{aligned}
\]
Transposing the linear maps $J^{t,m}_s$ and $\overline J^{t,m}_s$ gives the
vector formula. Note that the linear functional derivative itself is defined up to an
additive constant in $x$, as usual. Its spatial gradient is intrinsic.
\end{proof}

\begin{remark}
\leavevmode
\begin{enumerate}
\item[(i)]
The second term in the formula for
$\nabla_x\delta_m\Gr(s,t,m)$ has no finite-dimensional analogue. It is the
new contribution created by the dependence of the dynamics on the conditional
law. Indeed, when one perturbs the initial position of a tagged particle, the
perturbation has two effects. First, it changes the trajectory of this tagged
particle itself; this gives the diagonal tangent process $J^{t,m}$. Second,
because the drift depends on the conditional distribution, the perturbation of
the tagged particle also changes the conditional empirical measure seen by all
the other particles. This second effect propagates through the population and
is described, in the limit, by the off-diagonal tangent process
$\overline J^{t,m}$.

Thus the derivative of the propagated functional is not only the derivative of
the terminal functional along a single transported particle. It also contains
the derivative of the McKean--Vlasov environment generated by this particle.
This is precisely the role of the off-diagonal term.

\medskip

\item[(ii)]
One may also write an expression for a linear functional derivative
$\delta_m\Gr(s,t,m)$ itself, and not only for its spatial gradient. This
requires introducing the first-order variation of a generic particle with
respect to an infinitesimal perturbation of the initial law.

Fix a tagged trajectory $S^{t,m}$ starting from $U$, and let
$\overline S^{t,m}$ be a conditionally independent copy, given the common
noise. We define the process $\overline Y^{t,m}$ by
\[
\overline Y^{t,m}_r=0,\qquad r\in[0,t],
\]
and, for $r\in[t,s]$,
\[
\begin{aligned}
\mathrm{d}\overline Y^{t,m}_r
=
\nabla_xB(r,\overline S^{t,m}_r,\mu^{t,m}_r)
\overline Y^{t,m}_r\,\mathrm{d}r
&+
\delta_mB(r,\overline S^{t,m}_r,\mu^{t,m}_r)(S^{t,m}_r)\,\mathrm{d}r
\\
&+
\int_{\R^d\times\R^d}
\nabla_y\delta_mB(r,\overline S^{t,m}_r,\mu^{t,m}_r)(x)y\,
\widehat\nu^{t,m}_r(\mathrm{d}x,\mathrm{d}y)\,\mathrm{d}r,
\end{aligned}
\]
where
\[
\widehat\nu^{t,m}_r
:=
\Lc\bigl(\overline S^{t,m}_r,\overline Y^{t,m}_r
\,\big|\,S^{t,m},\Gc_r\bigr).
\]
The interpretation is the following. The term
\[
\delta_mB(r,\overline S^{t,m}_r,\mu^{t,m}_r)(S^{t,m}_r)
\]
is the direct effect of adding an infinitesimal mass at the tagged position
$S^{t,m}_r$. The last term accounts for the propagation of this perturbation
through the rest of the conditional population.

With this notation, a representative of the linear functional derivative of
\[
m\longmapsto \Gr(s,t,m)=G(\mu^{t,m}_s)
\]
is given, up to the usual additive normalization, by
\[
\begin{aligned}
\delta_m\Gr(s,t,m)(x)
&=
\E\Big[
\delta_mG(\mu^{t,m}_s)(S^{t,m}_s)
\,\Big|\,\Gc_s,\ U=x
\Big]
\\
&\quad+
\E\Big[
\nabla_x\delta_mG(\mu^{t,m}_s)(\overline S^{t,m}_s)
\cdot \overline Y^{t,m}_s
\,\Big|\,\Gc_s,\ U=x
\Big].
\end{aligned}
\]
The first term is the direct variation of the terminal functional $G$ due to
the tagged particle. The second term is the indirect variation coming from the
change in the conditional McKean--Vlasov flow generated by this infinitesimal
perturbation of the initial law.

Formally differentiating this expression with respect to the point $x$ gives
back the formula of \Cref{cor:derivative_propagated}. In particular, the
spatial derivative of the first term yields the diagonal contribution involving
$J^{t,m}$, while the spatial derivative of the second term yields the
off-diagonal contribution involving $\overline J^{t,m}$.
\end{enumerate}
\end{remark}

\section{Backward flow formula along an admissible conditional law} \label{sec:backward_flow}

Let $U \perp (W,W^\circ)$, $e \in [0,T]$ and $(\nu_r)_{e\le r\le T}$ be an admissible conditional law flow generated by
a bounded Lipschitz drift $\beta$. More precisely, let $X$ solve: $X_{r}=U$ for $r \in [0,e]$, and for $r \in [e,T]$,
\[
\mathrm{d}X_r
=
\beta(r,X_r,\nu_r)\,\mathrm{d}r
+\mathrm{d}W_r+\sigma_\circ\,\mathrm{d}W^\circ_r,
\]
with
\[
\nu_r:=\Lc(X_r\mid\Gc_r).
\]
Recall that
\[
\overline\Gr(s,t,m)
=
\E\bigl[G(\mu^{t,m}_s)\bigr].
\]
This deterministic functional is the Wasserstein kernel between times $t$ and
$s$.

\begin{theorem}[Backward flow formula]
\label{thm:backward_flow}
Assume {\rm \Cref{assum:main,assum:G}}. Let $0 \le e\le t\le s\le T$. Then, $\P$-a.s.,
\[
\begin{aligned}
\overline\Gr(s,t,\nu_t)
&=
G(\nu_s)
+
\int_t^s
\E\Big[
\nabla_x\delta_m\overline\Gr(s,r,\nu_r)(X_r)
\cdot
\bigl(B(r,X_r,\nu_r)-\beta(r,X_r,\nu_r)\bigr)
\,\Big|\,\Gc_r
\Big]\mathrm{d}r
\\
&\quad
-
\int_t^s
\E\Big[
\nabla_x\delta_m\overline\Gr(s,r,\nu_r)(X_r)
\,\Big|\,\Gc_r
\Big]\cdot\sigma_\circ\,\mathrm{d}W^\circ_r.
\end{aligned}
\]
In particular, with $s=T$,
\[
\begin{aligned}
\Kc_BG(t,\nu_t)
&=
G(\nu_T)
+
\int_t^T
\E\Big[
\nabla_x\delta_m\Kc_BG(r,\nu_r)(X_r)
\cdot
\bigl(B(r,X_r,\nu_r)-\beta(r,X_r,\nu_r)\bigr)
\,\Big|\,\Gc_r
\Big]\mathrm{d}r
\\
&\quad
-
\int_t^T
\E\Big[
\nabla_x\delta_m\Kc_BG(r,\nu_r)(X_r)
\,\Big|\,\Gc_r
\Big]\cdot\sigma_\circ\,\mathrm{d}W^\circ_r.
\end{aligned}
\]
\end{theorem}

\begin{proof}
The proof follows the same philosophy as the proof of
\Cref{prop:integral_formula_flow}, but now we also use a dynamic argument. The
functional $\overline\Gr(s,r,\cdot)$ is built from the reference dynamics with
drift $B$. Therefore, when it is evaluated along a flow driven by another
drift $\beta$, the only finite-variation discrepancy is the difference
$B-\beta$. The common-noise martingale remains visible because the state
variable is the conditional law.

\medskip

\noindent
\textbf{Step 1. Particle approximation of the admissible flow.}

Let $(X^i,W^i,U^i)_{i\ge1}$ be conditionally independent copies of $(X,W,U)$ given
the common noise $W^\circ$. Define
\[
\nu^n_r:=\frac1n\sum_{i=1}^n\delta_{X^i_r}.
\]
At the particle level, for $e \le r$,
\[
\mathrm{d}X^i_r
=
\beta(r,X^i_r,\nu^n_r)\,\mathrm{d}r
+\mathrm{d}W^i_r+\sigma_\circ\,\mathrm{d}W^\circ_r.
\]
Under the usual assumptions, conditional propagation of chaos gives
\[
\nu^n_r\Longrightarrow\nu_r,
\qquad r\in[e,T],
\]
conditionally on the common noise.

\medskip

\noindent
\textbf{Step 2. The reference finite-dimensional kernel.}

Fix $0\le r\le s$ and $\xbb=(x_1,\dots,x_n)\in(\R^d)^n$. Starting from
$\xbb$ at time $r$, define the reference $B$-particle system by
\[
S^{i,n,r,\xbb}_r=x_i,
\]
and
\[
\mathrm{d}S^{i,n,r,\xbb}_\ell
=
B(\ell,S^{i,n,r,\xbb}_\ell,\mu^{n,r,\xbb}_\ell)\,\mathrm{d}\ell
+\mathrm{d}W^i_\ell+\sigma_\circ\,\mathrm{d}W^\circ_\ell,
\qquad
\ell\in[r,s],
\]
where
\[
\mu^{n,r,\xbb}_\ell
:=
\frac1n\sum_{i=1}^n\delta_{S^{i,n,r,\xbb}_\ell}.
\]
Define
\[
\Gr^n(s,r,\xbb)
:=
\E\Big[
G\bigl(\mu^{n,r,\xbb}_s\bigr)
\Big].
\]
At the terminal time,
\[
\Gr^n(s,s,\xbb)
=
G\left(\frac1n\sum_{i=1}^n\delta_{x_i}\right).
\]

\medskip

\noindent
\\textbf{Step 3. Reference martingale property.}

Because $\Gr^n$ is generated by the $B$-particle dynamics, it solves the
backward Kolmogorov equation associated with the $n$-particle generator whose
drift is $B$. In particular, writing $\Ubb^n:=(U^1,\dots,U^n)$ and
$\Sbb^n:=\Sbb^{n,0,\Ubb^n}$,
\begin{align*}
    \Gr^n(s,r,\Sbb^{n}_r)
    =
    \E\Big[
    G\left(\mu^{n,0,\Ubb^n}_s\right)\,\Big|\,\Fc^{n}_r
    \Big],
\end{align*}
where, for $e\in[0,s]$,
\begin{align*}
\F^{e,n}:=(\Fc^{e,n}_r)_{r \in [e,s]},
\qquad
\Fc^{e,n}_r:=\sigma\{\, U^i,\; W^i_u-W^i_{e},\; W^\circ_u-W^\circ_{e}
\;:\; i \le n,\; u \in [e,r]\,\},
\end{align*}
and $\F^n:=\F^{0,n}$. Equivalently, if $\Sbb^n$ follows the $B$-particle
dynamics, then
\[
r\longmapsto \Gr^n(s,r,\Sbb^n_r)
\]
is an $\F^n$-martingale. The same argument, applied on $[e,s]$ with initial
condition $\Ubb^n$ at time $e$, shows that
\[
\left(\Gr^n\left(s,r,\Sbb^{n,e,\Ubb^n}_r\right)\right)_{r\in[e,s]}
\]
is an $\F^{e,n}$-martingale. Therefore, by martingale representation, for
$r\in[e,s]$,
\[
\mathrm{d}\Gr^n\left(s,r,\Sbb^{n,e,\Ubb^n}_r\right)
=
\sum_{i=1}^n
\nabla_{x_i}\Gr^n\left(s,r,\Sbb^{n,e,\Ubb^n}_r\right)
\cdot
\left(\mathrm{d}W^i_r+\sigma_\circ\,\mathrm{d}W^\circ_r\right).
\]

We now pass from the reference dynamics with drift $B$ to the dynamics with
drift $\beta$ by a Girsanov change of measure. More precisely, set
\[
\mathrm{d}\P^n:=Z^n_s\,\mathrm{d}\P,
\]
where $Z^n_{r\wedge e}=1$ and, for $r\in[e,s]$,
\[
\mathrm{d}Z^n_r
=
Z^n_r
\sum_{i=1}^n
\bigl(
\beta(r,S^{i,n,e,\Ubb^n}_r,\mu^{n,e,\Ubb^n}_r)
-
B(r,S^{i,n,e,\Ubb^n}_r,\mu^{n,e,\Ubb^n}_r)
\bigr)
\cdot \mathrm{d}W^i_r .
\]
Then, under $\P^n$,
\[
\widetilde W^{i,n}_r
:=
W^i_r
-
\int_e^r
\bigl(
\beta(u,S^{i,n,e,\Ubb^n}_u,\mu^{n,e,\Ubb^n}_u)
-
B(u,S^{i,n,e,\Ubb^n}_u,\mu^{n,e,\Ubb^n}_u)
\bigr)\,\mathrm{d}u
\]
is a Brownian motion on $[e,s]$. Hence the reference particles
$\Sbb^{n,e,\Ubb^n}$ have, under $\P^n$, the same dynamics as the particles
driven by the drift $\beta$. Equivalently, evaluating the same identity along
$\Xbb^n_r=(X^1_r,\dots,X^n_r)$, with empirical measure $\nu^n_r$, we obtain
\[
\begin{aligned}
\mathrm{d}\Gr^n(s,r,\Xbb^n_r)
&=
\sum_{i=1}^n
\nabla_{x_i}\Gr^n(s,r,\Xbb^n_r)
\cdot
\bigl(
\beta(r,X^i_r,\nu^n_r)
-
B(r,X^i_r,\nu^n_r)
\bigr)\,\mathrm{d}r
\\
&\quad+
\sum_{i=1}^n
\nabla_{x_i}\Gr^n(s,r,\Xbb^n_r)\cdot\mathrm{d}W^i_r
\\
&\quad+
\sum_{i=1}^n
\nabla_{x_i}\Gr^n(s,r,\Xbb^n_r)
\cdot\sigma_\circ\,\mathrm{d}W^\circ_r .
\end{aligned}
\]
Integrating from $t \in [e,s]$ to $s$ and using the terminal identity gives
\[
\begin{aligned}
\Gr^n(s,t,\Xbb^n_t)
&=
G(\nu^n_s)
+
\int_t^s
\sum_{i=1}^n
\nabla_{x_i}\Gr^n(s,r,\Xbb^n_r)
\cdot
\bigl(B(r,X^i_r,\nu^n_r)-\beta(r,X^i_r,\nu^n_r)\bigr)\,\mathrm{d}r
\\
&\quad
-
\sum_{i=1}^n
\int_t^s
\nabla_{x_i}\Gr^n(s,r,\Xbb^n_r)\cdot\mathrm{d}W^i_r
\\
&\quad
-
\int_t^s
\sum_{i=1}^n
\nabla_{x_i}\Gr^n(s,r,\Xbb^n_r)\cdot\sigma_\circ\,\mathrm{d}W^\circ_r.
\end{aligned}
\]

\medskip

\noindent
\textbf{Step 4. Conditioning on the common noise.}

We take conditional expectation with respect to the common noise. The stochastic
integral with respect to the idiosyncratic noises disappears. The common-noise
integral remains because it is measurable with respect to the common noise. We
obtain, for $t \in [e,s]$,
\[
\begin{aligned}
\E\bigl[\Gr^n(s,t,\Xbb^n_t)\mid\Gc_t\bigr]
&=
G(\nu^n_s)
+
\int_t^s
\E\Big[
\sum_{i=1}^n
\nabla_{x_i}\Gr^n(s,r,\Xbb^n_r)
\cdot
\bigl(B(r,X^i_r,\nu^n_r)-\beta(r,X^i_r,\nu^n_r)\bigr)
\,\Big|\,\Gc_r
\Big]\mathrm{d}r
\\
&\quad
-
\int_t^s
\E\Big[
\sum_{i=1}^n
\nabla_{x_i}\Gr^n(s,r,\Xbb^n_r)
\,\Big|\,\Gc_r
\Big]\cdot\sigma_\circ\,\mathrm{d}W^\circ_r.
\end{aligned}
\]
This is the finite-particle backward flow formula.

\medskip

\noindent
\textbf{Step 5. Identification of the particle gradients.}

For $h\in\R^d$, the same particle differentiation as before gives
\[
\nabla_{x_i}\Gr^n(s,r,\xbb)\cdot h
=
\frac1n\sum_{j=1}^n
\E\Big[
\nabla_x\delta_mG(\mu^{n,r,\xbb}_s)(S^{j,n,r,\xbb}_s)
\cdot
J^{j,i,n,r,\xbb}_s h
\Big].
\]
Splitting $j=i$ and $j\neq i$, and passing to the limit, we recover exactly
the Lions derivative of the propagated functional shown in the proof of \Cref{prop:integral_formula_flow} and \Cref{cor:derivative_propagated}. Hence
\[
\sum_{i=1}^n
\nabla_{x_i}\Gr^n(s,r,\Xbb^n_r)
\cdot
\bigl(B(r,X^i_r,\nu^n_r)-\beta(r,X^i_r,\nu^n_r)\bigr)
\]
converges to
\[
\E\Big[
\nabla_x\delta_m\overline\Gr(s,r,\nu_r)(X_r)
\cdot
\bigl(B(r,X_r,\nu_r)-\beta(r,X_r,\nu_r)\bigr)
\,\Big|\,\Gc_r
\Big].
\]

\medskip

\noindent
\textbf{Step 6. Passage to the limit.}

By conditional propagation of chaos,
\[
\nu^n_s\Longrightarrow\nu_s,
\]
and therefore
\[
G(\nu^n_s)\longrightarrow G(\nu_s).
\]
Moreover,
\[
\E\bigl[\Gr^n(s,t,\Xbb^n_t)\mid\Gc_t\bigr]
\longrightarrow
\overline\Gr(s,t,\nu_t).
\]
Passing to the limit in the finite-particle backward formula yields
\[
\begin{aligned}
\overline\Gr(s,t,\nu_t)
&=
G(\nu_s)
+
\int_t^s
\E\Big[
\nabla_x\delta_m\overline\Gr(s,r,\nu_r)(X_r)
\cdot
\bigl(B(r,X_r,\nu_r)-\beta(r,X_r,\nu_r)\bigr)
\,\Big|\,\Gc_r
\Big]\mathrm{d}r
\\
&\quad
-
\int_t^s
\E\Big[
\nabla_x\delta_m\overline\Gr(s,r,\nu_r)(X_r)
\,\Big|\,\Gc_r
\Big]\cdot\sigma_\circ\,\mathrm{d}W^\circ_r.
\end{aligned}
\]
This proves the theorem.
\end{proof}

\paragraph*{Interpretation as a McKean--Vlasov semigroup}

\Cref{thm:backward_flow} shows that the operator
\[
G\longmapsto \Kc_BG
\]
plays the same role on $\Pc(\R^d)$ as the semigroup does on $\R^d$.
The analogy can be summarized as follows.

In finite dimension, if
\[
v(t,x)=\E[g(X^{t,x}_T)]
\]
is generated by a reference diffusion with drift $B$, then along another
diffusion driven by $\beta$,
\[
\mathrm{d}v(t,X_t)
=
\nabla_xv(t,X_t)\cdot(\beta(t,X_t)-B(t,X_t))\,\mathrm{d}t
+
\nabla_xv(t,X_t)\cdot\mathrm{d}W_t.
\]
Only the first spatial derivative appears.
In the Wasserstein setting, if
\[
\Kc_BG(t,m)=\E[G(\mu^{t,m}_T)],
\]
then along a conditional law flow $(\nu_t)$ driven by $\beta$,
\[
\begin{aligned}
\mathrm{d}\Kc_BG(t,\nu_t)
&=
\E\Big[
\nabla_x\delta_m\Kc_BG(t,\nu_t)(X_t)
\cdot
(\beta(t,X_t,\nu_t)-B(t,X_t,\nu_t))
\,\Big|\,\Gc_t
\Big]\mathrm{d}t
\\
&\quad+
\E\Big[
\nabla_x\delta_m\Kc_BG(t,\nu_t)(X_t)
\,\Big|\,\Gc_t
\Big]\cdot\sigma_\circ\,\mathrm{d}W^\circ_t.
\end{aligned}
\]
Again, only the first Lions derivative appears. The time derivative and the
second--order Lions derivatives are hidden inside the kernel construction.

\paragraph*{Feynman--Kac interpretation.}

The McKean--Vlasov semigroup provides a Feynman--Kac representation for linear
parabolic equations on the space of probability measures, in direct analogy
with the classical finite-dimensional setting. Let
\[
B:[0,T]\times\R^d\times\Pc(\R^d)\longrightarrow\R^d
\]
be a bounded drift. For a sufficiently smooth map
$\varphi:[0,T]\times\Pc(\R^d)\to\R$, define
\[
\begin{aligned}
\Lgen^B\varphi(t,m)
&:=
\int_{\R^d}
\nabla_x\delta_m\varphi(t,m)(x)\cdot B(t,x,m)\,m(\mathrm dx)
\\
&\quad+
\frac12\int_{\R^d}
\trace \left[
\nabla_x^2\delta_m\varphi(t,m)(x)
\bigl(\Ir_d+\sigma_\circ\sigma_\circ^\top\bigr)
\right]m(\mathrm dx)
\\
&\quad+
\frac12\int_{\R^d\times\R^d}
\trace \left[
\nabla_y\nabla_x\delta_m^2\varphi(t,m)(x,y)
\sigma_\circ\sigma_\circ^\top
\right]m(\mathrm dx)m(\mathrm dy).
\end{aligned}
\]

We call $\varphi$ smooth if it is continuously differentiable in time,
admits first and second linear functional derivatives in the measure variable,
and the maps
\[
\partial_t\varphi,\qquad
\nabla_x\delta_m\varphi,\qquad
\nabla_x^2\delta_m\varphi,\qquad
\nabla_y\nabla_x\delta_m^2\varphi
\]
are continuous and bounded on their respective domains.

Consider the terminal--value problem
\begin{equation}
\label{eq:Wasserstein_heat_equation}
\left\{
\begin{aligned}
\partial_t\Uc(t,m)+\Lgen^B\Uc(t,m)&=0,
&& (t,m)\in[0,T)\times\Pc(\R^d),
\\
\Uc(T,m)&=G(m),
&&m\in\Pc(\R^d).
\end{aligned}
\right.
\end{equation}
Let $(\mu_r^{t,m})_{r\in[t,T]}$ be the conditional McKean--Vlasov flow
generated by $B$. If $\Uc$ is a smooth solution of
\eqref{eq:Wasserstein_heat_equation}, then the Itô formula along
$(\mu_r^{t,m})_{r\in[t,T]}$ yields
\[
\mathrm d\Uc(r,\mu_r^{t,m})
=
\left(
\int_{\R^d}
\nabla_x\delta_m\Uc(r,\mu_r^{t,m})(x)\,
\mu_r^{t,m}(\mathrm dx)
\right)
\cdot\sigma_\circ\,\mathrm dW_r^\circ.
\]
Hence $\bigl(\Uc(r,\mu_r^{t,m})\bigr)_{r\in[t,T]}$ is a martingale and,
using the terminal condition,
\[
\Uc(t,m)
=
\E\bigl[\Uc(T,\mu_T^{t,m})\bigr]
=
\E\bigl[G(\mu_T^{t,m})\bigr]
=
\Kc_BG(t,m).
\]
Therefore, any smooth solution of
\eqref{eq:Wasserstein_heat_equation} is necessarily represented by the
McKean--Vlasov semigroup:
\[
\Uc(t,m)=\Kc_BG(t,m),
\qquad
(t,m)\in[0,T]\times\Pc(\R^d).
\]

Conversely, whenever the map
\[
(t,m)\longmapsto\Kc_BG(t,m)
\]
is sufficiently smooth for the Wasserstein Itô formula to apply, the martingale
property (see \Cref{sec:martingale_representation} below) of
\[
r\longmapsto\Kc_BG(r,\mu_r^{t,m})
\]
shows that $\Kc_BG$ solves
\eqref{eq:Wasserstein_heat_equation}. Thus the McKean--Vlasov semigroup is the
natural Feynman--Kac representation of the linear Wasserstein parabolic
equation generated by $\Lgen^B$.

\begin{remark}[The driftless reference case]
Assume that the reference drift is $B=0$. Then the backward flow formula becomes
\begin{align} \label{eq:B=0}
\Kc_0G(t,\nu_t)
=
G(\nu_T)
&-
\int_t^T
\E\Big[
\nabla_x\delta_m\Kc_0G(r,\nu_r)(X_r)
\cdot
\beta(r,X_r,\nu_r)
\,\Big|\,\Gc_r
\Big]\mathrm{d}r
\nonumber
\\
&\quad-
\int_t^T
\E\Big[
\nabla_x\delta_m\Kc_0G(r,\nu_r)(X_r)
\,\Big|\,\Gc_r
\Big]\cdot\sigma_\circ\,\mathrm{d}W^\circ_r .
\end{align}
This is the closest analogue of the classical heat-semigroup representation. The
kernel $\Kc_0G$ is constructed from the conditional Brownian flow, while the
identity above holds along an arbitrary admissible conditional law flow driven
by a drift $\beta$. The drift term in \eqref{eq:B=0} is therefore exactly the
correction generated by evaluating the Brownian kernel along a non-Brownian
flow.
\end{remark}

\section{A martingale representation formulation} \label{sec:martingale_representation}

In the driftless reference case $B=0$, the identity \eqref{eq:B=0} can be
viewed as a martingale representation theorem on the space of probability
measures. The role of the unknown process $Y$ is played by the value of the
semigroup along the conditional law flow, while the role of the martingale
integrand is played by the first-order Lions derivative of the semigroup.
We now formulate this idea intrinsically. The point is to characterize the pair
\[
\bigl(\Kc_0G,\nabla_x\delta_m\Kc_0G\bigr)
\]
as the unique pair satisfying the backward identity along all admissible
conditional law flows.

\begin{definition}[Martingale representation problem]
\label{def:martingale_representation_problem}
Let $G:\Pc(\R^d)\to\R$ be given. A pair of Borel maps
\[
Y:[0,T]\times\Pc(\R^d)\to\R,
\qquad
Z:[0,T]\times\R^d\times\Pc(\R^d)\to\R^d
\]
is called a solution of the martingale representation problem associated with
$G$ if the following property holds.

For every initial time $e\in[0,T]$, every initial law $m\in\Pc(\R^d)$,
and every admissible conditional law flow $(\nu_r)_{r\in[e,T]}$ generated by
a bounded Lipschitz drift $\beta$, namely
\[
\mathrm{d}X_r
=
\beta(r,X_r,\nu_r)\,\mathrm{d}r
+\mathrm{d}W_r+\sigma_\circ\,\mathrm{d}W^\circ_r,
\qquad
\nu_r=\Lc(X_r\mid\Gc_r),
\qquad
\nu_e=m,
\]
the following integrability condition is satisfied:
\[
\E\left[
\sup_{r\in[e,T]} |Y(r,\nu_r)|
+
\int_e^T
\left(
|Z(r,X_r,\nu_r)|
+
\left|
\E[Z(r,X_r,\nu_r)\mid\Gc_r]
\right|^2
\right)\mathrm{d}r
\right]
<\infty,
\]
and, $\P$-a.s., for every $t\in[e,T]$,
\[
\begin{aligned}
Y(t,\nu_t)
=
G(\nu_T)
-
\int_t^T
\E\Big[
Z(r,X_r,\nu_r)\cdot\beta(r,X_r,\nu_r)
\,\Big|\,\Gc_r
\Big]\mathrm{d}r
-
\int_t^T
\E\Big[
Z(r,X_r,\nu_r)
\,\Big|\,\Gc_r
\Big]\cdot\sigma_\circ\,\mathrm{d}W^\circ_r .
\end{aligned}
\]
\end{definition}

\begin{definition}[Uniqueness]
We say that the martingale representation problem is unique if, whenever
$(Y^1,Z^1)$ and $(Y^2,Z^2)$ are two solutions, then
\[
Y^1(t,m)=Y^2(t,m),
\qquad
\text{for all } (t,m)\in[0,T]\times\Pc(\R^d),
\]
and
\[
Z^1(t,x,m)=Z^2(t,x,m),
\qquad
m(\mathrm{d}x)\text{-a.e.},
\]
for every $(t,m)\in[0,T]\times\Pc(\R^d)$.
\end{definition}

\begin{proposition}[Martingale representation by the Wasserstein heat semigroup]
\label{prop:martingale_representation_kernel}
Assume that $G:\Pc(\R^d)\to\R$ is bounded continuous and admits a bounded linear functional
derivative $\delta_mG$ such that for each $x \in \R^d$,
\[
m\longmapsto \delta_mG(m)(x)
\]
is continuous. Then the martingale
representation problem associated with $G$ admits a unique solution, given by
\[
Y(t,m)=\Kc_0G(t,m),
\qquad
Z(t,x,m)=\nabla_x\delta_m\Kc_0G(t,m)(x).
\]
\end{proposition}

\begin{proof}
We split the proof into two parts: existence and uniqueness.

\medskip

\noindent
\textbf{Step 1. Existence.}

By the backward flow formula in the case $B=0$, the pair
\[
Y(t,m):=\Kc_0G(t,m),
\qquad
Z(t,x,m):=\nabla_x\delta_m\Kc_0G(t,m)(x)
\]
satisfies, along every admissible flow $(\nu_r)_{r\in[e,T]}$ driven by a drift
$\beta$,
\[
\begin{aligned}
\Kc_0G(t,\nu_t)
=
G(\nu_T)
-
\int_t^T
\E\Big[
\nabla_x\delta_m\Kc_0G(r,\nu_r)(X_r)
\cdot
\beta(r,X_r,\nu_r)
\,\Big|\,\Gc_r
\Big]\mathrm{d}r
-
\int_t^T
\E\Big[
\nabla_x\delta_m\Kc_0G(r,\nu_r)(X_r)
\,\Big|\,\Gc_r
\Big]\cdot\sigma_\circ\,\mathrm{d}W^\circ_r .
\end{aligned}
\]
Thus $(Y,Z)$ is a solution of the martingale representation problem.
It remains to explain the integrability condition and why, in the driftless reference case, the derivative of
the kernel can be constructed from the linear functional derivative of $G$,
without assuming the terminal Lions derivative
$\nabla_x\delta_mG$. Let
\[
M^{t,m}_T:=\Lc(N^{t,U}_T\mid\Gc_T),
\]
where
\[
N^{t,U}_r=U,\qquad r\in[0,t],
\]
and, for $r\in[t,T]$,
\[
\mathrm{d}N^{t,U}_r
=
\mathrm{d}W_r+\sigma_\circ\,\mathrm{d}W^\circ_r.
\]
For a point $x\in\R^d$, write $N^{t,x}$ for the same process starting from
$x$. The Brownian smoothing gives, for $t<T$,
\[
\nabla_x\delta_m\Kc_0G(t,m)(x)
=
\nabla_x
\E\Big[
\delta_mG(M^{t,m}_T)(N^{t,x}_T)
\Big].
\]
By the Gaussian integration-by-parts formula, for each coordinate $i=1,\dots,d$,
\[
\begin{aligned}
\nabla_{x_i}\delta_m\Kc_0G(t,m)(x)
=
\E\Big[
\nabla_{x_i}\delta_mG(M^{t,m}_T)(N^{t,x}_T)
\Big]
=
\E\Big[
\delta_mG(M^{t,m}_T)(N^{t,x}_T)
\frac{(W_T^i-W_t^i)}{T-t}
\Big],
\end{aligned}
\]
whenever the first equality is justified. The second equality, however, still makes
sense under the weaker assumption that only $\delta_mG$ exists. Hence, by an approximation argument, the identity defines
$\nabla_x\delta_m\Kc_0G$ for $t<T$ without requiring
$\nabla_x\delta_mG$ at the terminal level.

It remains to check the integrability condition in the definition of the
martingale representation problem. Since $G$ is bounded, the kernel is bounded
as well:
\[
|Y(t,m)|
=
|\Kc_0G(t,m)|
\leq \|G\|_\infty .
\]
Next, the Gaussian integration-by-parts formula gives the estimate
\[
|Z(t,x,m)|
=
|\nabla_x\delta_m\Kc_0G(t,m)(x)|
\leq
\frac{C}{\sqrt{T-t}},
\qquad 0\leq t<T,
\]
for a constant $C$ depending only on the bound of $\delta_mG$ and on the
dimension. Indeed, for each coordinate $i$,
\[
\nabla_{x_i}\delta_m\Kc_0G(t,m)(x)
=
\E\left[
\delta_mG(M^{t,m}_T)(N^{t,x}_T)
\frac{W_T^i-W_t^i}{T-t}
\right],
\]
and therefore
\[
\left|
\nabla_{x_i}\delta_m\Kc_0G(t,m)(x)
\right|
\leq
\|\delta_mG\|_\infty
\frac{\E[|W_T^i-W_t^i|]}{T-t}
\leq
\frac{C}{\sqrt{T-t}} .
\]
Consequently,
\[
\E\left[\int_0^T |Z(t,X_t,\nu_t)|\,\mathrm{d}t\right]
\leq
C\int_0^T \frac{\mathrm{d}t}{\sqrt{T-t}}
<\infty .
\]

It remains to verify the square-integrability of the common-noise integrand
\[
\E\bigl[Z(t,X_t,\nu_t)\mid\Gc_t\bigr]\cdot\sigma_\circ .
\]
For this, we use the backward representation already obtained for $Y$. Along
an admissible flow driven by $\beta$, one has
\[
\mathrm{d}Y(t,\nu_t)
=
\E\Big[
Z(t,X_t,\nu_t)\cdot \beta(t,X_t,\nu_t)
\,\Big|\,\Gc_t
\Big]\mathrm{d}t
+
\E\Big[
Z(t,X_t,\nu_t)
\,\Big|\,\Gc_t
\Big]\cdot\sigma_\circ\,\mathrm{d}W^\circ_t .
\]
Applying Itô's formula to $|Y(t,\nu_t)|^2$, we get
\[
\begin{aligned}
|Y(t,\nu_t)|^2
&=
|Y(0,\nu_0)|^2
+
2\int_0^t
Y(r,\nu_r)
\E\Big[
Z(r,X_r,\nu_r)\cdot \beta(r,X_r,\nu_r)
\,\Big|\,\Gc_r
\Big]\mathrm{d}r
\\
&\quad+
2\int_0^t
Y(r,\nu_r)
\E\Big[
Z(r,X_r,\nu_r)
\,\Big|\,\Gc_r
\Big]\cdot\sigma_\circ\,\mathrm{d}W^\circ_r
+
\int_0^t
\left|
\sigma_\circ^\top
\E\Big[
Z(r,X_r,\nu_r)
\,\Big|\,\Gc_r
\Big]
\right|^2\mathrm{d}r .
\end{aligned}
\]
The stochastic integral is a true martingale after localization, and the
localization can be removed by the boundedness of $Y$ and the integrability
estimate on $Z$. Taking expectations therefore yields
\[
\begin{aligned}
\E\left[
\int_0^t
\left|
\sigma_\circ^\top
\E\Big[
Z(r,X_r,\nu_r)
\,\Big|\,\Gc_r
\Big]
\right|^2\mathrm{d}r
\right]
&\leq
\E|Y(t,\nu_t)|^2
+
\E|Y(0,\nu_0)|^2
\\
&\quad+
2\E\left[
\int_0^t
|Y(r,\nu_r)|
\left|
\E\Big[
Z(r,X_r,\nu_r)\cdot\beta(r,X_r,\nu_r)
\,\Big|\,\Gc_r
\Big]
\right|
\mathrm{d}r
\right].
\end{aligned}
\]
Using the boundedness of $Y$ and $\beta$, together with
\[
\E\int_0^T |Z(r,X_r,\nu_r)|\,\mathrm{d}r<\infty,
\]
we obtain a bound independent of $t<T$. Letting $t\uparrow T$, by monotone
convergence,
\[
\E\left[
\int_0^T
\left|
\sigma_\circ^\top
\E\Big[
Z(r,X_r,\nu_r)
\,\Big|\,\Gc_r
\Big]
\right|^2\mathrm{d}r
\right]
<\infty .
\]
This proves the required square-integrability of the common-noise martingale
integrand.

\medskip

\noindent
\textbf{Step 2. Identification of $Y$.}

Let $(Y^i,Z^i)$, $i=1,2$, be two solutions. Fix an initial time
$e\in[0,T]$ and an initial law $m\in\Pc(\R^d)$. We first choose the
admissible flow corresponding to the drift $\beta=0$. Thus
\[
\nu_r=M^{e,m}_r,
\]
where $M^{e,m}$ is the conditional Brownian flow starting from $m$ at time
$e$. Applying the martingale representation identity at time $e$, we get
\[
Y^i(e,m)
=
G(M^{e,m}_T)
-
\int_e^T
\E\Big[
Z^i(r,N^{e,U}_r,M^{e,m}_r)
\,\Big|\,\Gc_r
\Big]\cdot\sigma_\circ\,\mathrm{d}W^\circ_r.
\]
Taking expectation yields
\[
Y^i(e,m)
=
\E\bigl[G(M^{e,m}_T)\bigr]
=
\Kc_0G(e,m).
\]
Since $e$ and $m$ are arbitrary, we obtain
\[
Y^1=Y^2=\Kc_0G.
\]

\medskip

\noindent
\textbf{Step 3. Identification of $Z$.}

We now compare the two representations. Since $Y^1=Y^2$, subtracting the two
identities gives, for every admissible flow $(\nu_r)$ driven by $\beta$,
\[
\begin{aligned}
0
=
\int_t^T
\E\Big[
\bigl(Z^1-Z^2\bigr)(r,X_r,\nu_r)
\cdot\beta(r,X_r,\nu_r)
\,\Big|\,\Gc_r
\Big]\mathrm{d}r
+
\int_t^T
\E\Big[
\bigl(Z^1-Z^2\bigr)(r,X_r,\nu_r)
\,\Big|\,\Gc_r
\Big]\cdot\sigma_\circ\,\mathrm{d}W^\circ_r .
\end{aligned}
\]
The finite-variation part and the martingale part must vanish separately.
Therefore,
\[
\E\Big[
\bigl(Z^1-Z^2\bigr)(r,X_r,\nu_r)
\cdot\beta(r,X_r,\nu_r)
\,\Big|\,\Gc_r
\Big]=0
\]
for Lebesgue-a.e. $r$, and
\[
\E\Big[
\bigl(Z^1-Z^2\bigr)(r,X_r,\nu_r)
\,\Big|\,\Gc_r
\Big]\cdot\sigma_\circ=0
\]
for Lebesgue-a.e. $r$. To recover pointwise uniqueness of $Z$, it is enough
to use the first identity and the arbitrariness of the drift $\beta$.
Indeed, fix $e\in[0,T]$ and $m\in\Pc(\R^d)$. Choose $Z^2(t,x,m)=\nabla_x \delta_m \Kc_0G(t,m)(x)$ and a smooth bounded drift
$\beta$. For such drifts, the map
\[
r\longmapsto
\E\Big[
Z^1(r,X_r,\nu_r)
\cdot\beta(r,X_r,\nu_r)
\Big]= \E\Big[Z^2(r,X_r,\nu_r)
\cdot\beta(r,X_r,\nu_r)
\Big]
\]
can be taken continuous, since $r \mapsto Z^2(r,X_r,\nu_r)$ is continuous. Hence the a.e. equality may be evaluated at $r=e$. Since
$\nu_e=m$ and $X_e\sim m$, we obtain
\[
\int_{\R^d}
\bigl(Z^1(e,x,m)-Z^2(e,x,m)\bigr)
\cdot\beta(e,x,m)\,m(\mathrm{d}x)
=0.
\]
The value of $\beta(e,\cdot,m)$ can be chosen as an arbitrary smooth bounded
test vector field. Therefore,
\[
Z^1(e,x,m)=Z^2(e,x,m),
\qquad
m(\mathrm{d}x)\text{-a.e.}
\]
Since $e$ and $m$ are arbitrary, this proves uniqueness of $Z$.
The proof is complete.
\end{proof}

\section{A mixed finite-dimensional/Wasserstein formula} \label{sec:mixed_formula}

We finally record a useful extension in which the terminal functional depends
both on a finite-dimensional state and on a measure argument. This type of
formula is useful when one studies systems containing both a representative
state variable and a conditional distribution.

\smallskip
Let $\widehat B:[0,T]\times\R^d\to\R^d$
be a bounded map,  Lipschitz in space uniformly in time. For $(t,p)\in[0,T]\times\R^d$, define $\widehat S^{t,p}_r=p$, $r\in[0,t],$
and
\[
\mathrm{d}\widehat S^{t,p}_r
=
\widehat B(r,\widehat S^{t,p}_r)\,\mathrm{d}r
+\mathrm{d}W_r+\sigma_\circ\,\mathrm{d}W^\circ_r.
\]

Let $\widehat G:\R^d\times\Pc(\R^d)\to\R$ be a bounded map such that $m \mapsto \widehat{G}(x,m)$ satisfies \Cref{assum:G}
uniformly in $x$, and define
\[
\widehat\Gr(s,t,m,p)
:=
\E\Big[
\widehat G(\widehat S^{t,p}_s,\mu^{t,m}_s)
\Big].
\]
Let $e \in [0,s]$ be an initial time, $X_0$ be an $\R^d$--valued random variable independent of $(W,W^\circ)$, 
$X_\cdot:=X_0 + W_{\cdot \vee e} - W_e + \sigma_\circ \cdot (W^\circ_{\cdot \vee e} - W^\circ_e)$. 

\begin{proposition}[Mixed backward flow formula]
Let $(\nu_r)_{r \in [e,s]}$ be generated by the drift $\beta$, and let $\widehat X$ be
generated by $\widehat\beta$. Then, for $e\le t\le s\le T$,
\[
\begin{aligned}
\widehat\Gr(s,t,\nu_t,\widehat X_t)
&=
\widehat G(\widehat X_s,\nu_s)
\\
&\quad+
\int_t^s
\E\Big[
\nabla_x\delta_m\widehat\Gr(s,r,\nu_r,\widehat X_r)(X_r)
\cdot
\bigl(B(r,X_r,\nu_r)-\beta(r,X_r,\nu_r)\bigr)
\,\Big|\,\Gc_r
\Big]\mathrm{d}r
\\
&\quad+
\int_t^s
\nabla_p\widehat\Gr(s,r,\nu_r,\widehat X_r)
\cdot
\widehat B(r,\widehat X_r)\,\mathrm{d}r
\\
&\quad-
\int_t^s
\E\Big[
\nabla_x\delta_m\widehat\Gr(s,r,\nu_r,\widehat X_r)(X_r)
\,\Big|\,\Gc_r
\Big]\cdot\sigma_\circ\,\mathrm{d}W^\circ_r
\\
&\quad-
\int_t^s
\nabla_p\widehat\Gr(s,r,\nu_r,\widehat X_r)
\cdot
\bigl(\mathrm{d}W_r+\sigma_\circ\,\mathrm{d}W^\circ_r\bigr).
\end{aligned}
\]
\end{proposition}

\begin{proof}
The proof is a direct combination of the finite-dimensional and Wasserstein
arguments. The variable $p$ is propagated by the reference drift
$\widehat B$, while the measure variable is propagated by the reference
McKean--Vlasov drift $B$. 
The martingale part splits into two pieces: the direct noise acting on
$\widehat X$, and the common-noise fluctuation of the conditional law. This
gives the announced identity.
\end{proof}

\section{A first step toward a path-dependent framework}
\label{sec:path_dependent}

A natural extension of the martingale representation problem developed above
would consist in considering terminal functionals depending on the entire path
of the conditional law, namely,
\[
G:
C\bigl([0,T];\Pc(\R^d)\bigr)
\longrightarrow
\R,
\qquad
(m_t)_{t\in[0,T]}
\longmapsto
G\bigl((m_t)_{t\in[0,T]}\bigr).
\]
A complete treatment of this problem would require working directly on a
space of measure-valued paths and developing an appropriate notion of
differentiation with respect to their past trajectories. We leave this
general path-dependent framework for future work.

\medskip
The dynamic programming structure of the McKean--Vlasov semigroup nevertheless
allows us to accommodate a discrete form of path dependence. More precisely,
we consider functionals depending on the values of the conditional law at a
finite number of deterministic times. The resulting representation is
obtained by iterating the Markovian representation established in the previous
sections.

\subsection*{Discrete path-dependent functionals}

Let $k\geq 1$, and fix a deterministic grid
\[
0=t_0<t_1<\cdots<t_k=T.
\]
A discrete path-dependent terminal functional is a map
\[
G:
\Pc(\R^d)^{k+1}
\longrightarrow
\R.
\]
Given a measure-valued path
$
\boldsymbol m=(m_t)_{t\in[0,T]},
$
we use the notation
\[
G(\boldsymbol m)
:=
G(m_{t_0},\ldots,m_{t_k}).
\]

We say that $G$ admits bounded continuous linear functional derivatives if,
for every $j\in\{0,\ldots,k\}$, there exists a bounded map
\[
\delta_jG:
\Pc(\R^d)^{k+1}\times\R^d
\longrightarrow
\R
\]
which is continuous in $\Pc(\R^d)^{k+1}$ and such that, for every
$\boldsymbol m=(m_0,\ldots,m_k)\in\Pc(\R^d)^{k+1}$ and
$m'_j\in\Pc(\R^d)$,
\[
\begin{aligned}
&G(m_0,\ldots,m_{j-1},m'_j,m_{j+1},\ldots,m_k)
-
G(m_0,\ldots,m_{j-1},m_j,m_{j+1},\ldots,m_k)
\\
&\quad
=
\int_0^1
\int_{\R^d}
\delta_jG
\bigl(
    m_0,\ldots,m_{j-1},
    (1-\lambda)m_j+\lambda m'_j,
    m_{j+1},\ldots,m_k
\bigr)(x)
\,
(m'_j-m_j)(\mathrm dx)
\,\mathrm d\lambda.
\end{aligned}
\]
Thus, $\delta_jG$ denotes the linear functional derivative of $G$
with respect to its $j$-th measure argument.

Let
\[
\Cc_{\Pc}
:=
C\bigl([0,T];\Pc(\R^d)\bigr),
\]
endowed, for instance, with the topology of uniform convergence 
 induced by any metric
$\mathrm d_{\mathrm w}$ that metrizes the weak convergence on $\Pc(\R^d)$.

Let $E$ be a Polish space. We say that a Borel map
\[
F:[0,T]\times\Cc_{\Pc}\longrightarrow E
\]
is progressively measurable with respect to the grid
$(t_j)_{0\leq j\leq k}$ if, for every $j\in\{0,\ldots,k-1\}$, there
exists a Borel map
\[
F_j:
[t_j,t_{j+1})
\times
\Pc(\R^d)
\times
\Pc(\R^d)^{j+1}
\longrightarrow E
\]
such that
\[
F(t,\boldsymbol m)
=
F_j
\bigl(
    t,m_t;
    m_{t_0},\ldots,m_{t_j}
\bigr),
\qquad
t\in[t_j,t_{j+1}).
\]
In other words, between two consecutive grid times, the map may depend on the
current measure $m_t$ and on the values of the path observed at the previous
grid times.

Similarly, a drift
\[
B:
[0,T]\times\R^d\times\Cc_{\Pc}
\longrightarrow
\R^d
\]
is progressively measurable with respect to the grid if
\[
B(t,x,\boldsymbol m)
=
B_j
\bigl(
    t,x,m_t;
    m_{t_0},\ldots,m_{t_j}
\bigr),
\qquad
t\in[t_j,t_{j+1}),
\]
for suitable Borel maps $B_j$.

We say that a flow
\[
\boldsymbol\mu=(\mu_t)_{t\in[0,T]}
\]
is admissible if there exist a progressively measurable drift $B$, satisfying
the standing boundedness and Lipschitz assumptions, and an adapted process
$X$ such that
\[
\mu_t=\Lc(X_t\mid\Gc_t),
\qquad t\in[0,T],
\]
and
\[
\mathrm dX_t
=
B(t,X_t,\boldsymbol\mu)\,\mathrm dt
+
\mathrm dW_t
+
\sigma_\circ\,\mathrm dW^\circ_t.
\]
As usual, conditioning with respect to $\Gc_T$ instead of $\Gc_t$ gives
the same conditional law at time $t$ under the present filtration
assumptions.

\begin{proposition}
\label{prop:discrete_path_martingale_representation}
Assume that $G:\Pc(\R^d)^{k+1}\to\R$ is bounded and continuous and admits
bounded continuous linear functional derivatives with respect to each of its
measure arguments.

Then there exist maps
\[
Y:
[0,T]\times\Cc_{\Pc}
\longrightarrow
\R
\]
and
\[
Z:
[0,T]\times\R^d\times\Cc_{\Pc}
\longrightarrow
\R^d,
\]
progressively measurable with respect to the grid, such that
\[
Y(T,\boldsymbol m)
=
G(m_{t_0},\ldots,m_{t_k}).
\]
Moreover, for every admissible flow
$\boldsymbol\mu=(\mu_t)_{t\in[0,T]}$ associated with a process $X$,
and every $t\in[0,T]$,
\[
\begin{aligned}
Y(t,\boldsymbol\mu)
&=
G(\mu_{t_0},\ldots,\mu_{t_k})
\\
&\quad
-
\int_t^T
\E\left[
    Z(r,X_r,\boldsymbol\mu)
    \cdot
    B(r,X_r,\boldsymbol\mu)
    \,\middle|\,
    \Gc_r
\right]
\,\mathrm dr
\\
&\quad
-
\int_t^T
\E\left[
    Z(r,X_r,\boldsymbol\mu)
    \,\middle|\,
    \Gc_r
\right]
\cdot
\sigma_\circ\,\mathrm dW^\circ_r.
\end{aligned}
\]
\end{proposition}

\begin{proof}
The proof follows from a backward induction over the grid. At each step, the
values of the measure flow observed at the previous grid times are regarded
as frozen parameters, while the current measure remains the Markovian state
variable.

For $j\in\{0,\ldots,k-1\}$, write
\[
\boldsymbol m_j
:=
(m_0,\ldots,m_j)
\in
\Pc(\R^d)^{j+1}.
\]
We recursively construct terminal functions
\[
\Phi_j:
\Pc(\R^d)\times\Pc(\R^d)^{j+1}
\longrightarrow
\R
\]
and continuation values
\[
\Vc_j:
[t_j,t_{j+1}]
\times
\Pc(\R^d)
\times
\Pc(\R^d)^{j+1}
\longrightarrow
\R.
\]

We start from the last interval $[t_{k-1},t_k]$. For
$\boldsymbol m_{k-1}=(m_0,\ldots,m_{k-1})$, set
\[
\Phi_{k-1}(m;\boldsymbol m_{k-1})
:=
G(m_0,\ldots,m_{k-1},m).
\]
For every fixed history $\boldsymbol m_{k-1}$, let
\[
\Vc_{k-1}
\bigl(
    t,m;\boldsymbol m_{k-1}
\bigr),
\qquad
t\in[t_{k-1},t_k],
\]
be the Markovian continuation value associated with the terminal functional
\[
m\longmapsto\Phi_{k-1}(m;\boldsymbol m_{k-1}).
\]
Equivalently, using the driftless McKean--Vlasov semigroup with terminal time
$t_k$,
\[
\Vc_{k-1}
\bigl(
    t,m;\boldsymbol m_{k-1}
\bigr)
=
\Kc_0^{t_k}
\left[
    \Phi_{k-1}(\,\cdot\,;\boldsymbol m_{k-1})
\right](t,m).
\]
We also set, for $t<t_k$,
\[
\Zc_{k-1}
\bigl(
    t,x,m;\boldsymbol m_{k-1}
\bigr)
:=
\nabla_x\delta_m
\Vc_{k-1}
\bigl(
    t,m;\boldsymbol m_{k-1}
\bigr)(x).
\]

The Markovian martingale representation established above, applied
conditionally on the history
\[
(\mu_{t_0},\ldots,\mu_{t_{k-1}}),
\]
gives, for $t\in[t_{k-1},t_k]$,
\[
\begin{aligned}
&\Vc_{k-1}
\bigl(
    t,\mu_t;
    \mu_{t_0},\ldots,\mu_{t_{k-1}}
\bigr)
\\
&\quad
=
G(\mu_{t_0},\ldots,\mu_{t_k})
\\
&\qquad
-
\int_t^{t_k}
\E\left[
    \Zc_{k-1}
    \bigl(
        r,X_r,\mu_r;
        \mu_{t_0},\ldots,\mu_{t_{k-1}}
    \bigr)
    \cdot
    B(r,X_r,\boldsymbol\mu)
    \,\middle|\,
    \Gc_r
\right]
\,\mathrm dr
\\
&\qquad
-
\int_t^{t_k}
\E\left[
    \Zc_{k-1}
    \bigl(
        r,X_r,\mu_r;
        \mu_{t_0},\ldots,\mu_{t_{k-1}}
    \bigr)
    \,\middle|\,
    \Gc_r
\right]
\cdot
\sigma_\circ\,\mathrm dW^\circ_r.
\end{aligned}
\]

We now proceed backward. Suppose that, for some
$j\in\{0,\ldots,k-2\}$, the continuation value
\[
\Vc_{j+1}
\bigl(
    t,m;
    m_0,\ldots,m_{j+1}
\bigr)
\]
has already been constructed on $[t_{j+1},t_{j+2}]$. Define the terminal
functional on the preceding interval by
\[
\Phi_j(m;\boldsymbol m_j)
:=
\Vc_{j+1}
\bigl(
    t_{j+1},m;
    m_0,\ldots,m_j,m
\bigr).
\]
The measure $m$ appears twice in the right-hand side: it is both the current
measure at time $t_{j+1}$ and the newly recorded value of the discrete
history.

We next verify that $m\mapsto\Phi_j(m;\boldsymbol m_j)$ satisfies the
regularity assumptions required by the Markovian representation theorem.
The parameterized version of the semigroup regularity result implies that
$\Vc_{j+1}$ admits bounded continuous linear functional derivatives with
respect to its current measure argument and with respect to every component
of the frozen history.

Consequently, the chain rule for linear functional derivatives yields
\[
\begin{aligned}
\delta_m\Phi_j(m;\boldsymbol m_j)(x)
&=
\delta_m^{\mathrm{cur}}
\Vc_{j+1}
\bigl(
    t_{j+1},m;
    m_0,\ldots,m_j,m
\bigr)(x)
\\
&\quad
+
\delta_{j+1}^{\mathrm{hist}}
\Vc_{j+1}
\bigl(
    t_{j+1},m;
    m_0,\ldots,m_j,m
\bigr)(x),
\end{aligned}
\]
where the first term denotes the derivative with respect to the current
measure variable and the second one denotes the derivative with respect to
the last component of the discrete history. Both terms are bounded and
continuous. Hence $\Phi_j$ is again an admissible terminal functional.

For fixed $\boldsymbol m_j$, we may therefore define
\[
\Vc_j(t,m;\boldsymbol m_j)
=
\Kc_0^{t_{j+1}}
\left[
    \Phi_j(\,\cdot\,;\boldsymbol m_j)
\right](t,m),
\qquad
t\in[t_j,t_{j+1}],
\]
and
\[
\Zc_j(t,x,m;\boldsymbol m_j)
:=
\nabla_x\delta_m
\Vc_j(t,m;\boldsymbol m_j)(x),
\qquad
t<t_{j+1}.
\]

By construction,
\[
\Vc_j(t_{j+1},m;\boldsymbol m_j)
=
\Phi_j(m;\boldsymbol m_j)
=
\Vc_{j+1}
\bigl(
    t_{j+1},m;
    \boldsymbol m_j,m
\bigr).
\]
This is precisely the compatibility relation that connects two consecutive
time intervals.

Applying the Markovian representation on $[t_j,t_{j+1}]$, conditionally on
the frozen history
\[
(\mu_{t_0},\ldots,\mu_{t_j}),
\]
we obtain, for $t\in[t_j,t_{j+1}]$,
\[
\begin{aligned}
&\Vc_j
\bigl(
    t,\mu_t;
    \mu_{t_0},\ldots,\mu_{t_j}
\bigr)
\\
&\quad
=
\Phi_j
\bigl(
    \mu_{t_{j+1}};
    \mu_{t_0},\ldots,\mu_{t_j}
\bigr)
\\
&\qquad
-
\int_t^{t_{j+1}}
\E\left[
    \Zc_j
    \bigl(
        r,X_r,\mu_r;
        \mu_{t_0},\ldots,\mu_{t_j}
    \bigr)
    \cdot
    B(r,X_r,\boldsymbol\mu)
    \,\middle|\,
    \Gc_r
\right]
\,\mathrm dr
\\
&\qquad
-
\int_t^{t_{j+1}}
\E\left[
    \Zc_j
    \bigl(
        r,X_r,\mu_r;
        \mu_{t_0},\ldots,\mu_{t_j}
    \bigr)
    \,\middle|\,
    \Gc_r
\right]
\cdot
\sigma_\circ\,\mathrm dW^\circ_r.
\end{aligned}
\]
The compatibility relation gives
\[
\begin{aligned}
&\Phi_j
\bigl(
    \mu_{t_{j+1}};
    \mu_{t_0},\ldots,\mu_{t_j}
\bigr)
\\
&\qquad
=
\Vc_{j+1}
\bigl(
    t_{j+1},\mu_{t_{j+1}};
    \mu_{t_0},\ldots,\mu_{t_{j+1}}
\bigr).
\end{aligned}
\]

We now define the global maps $Y$ and $Z$. For
$t\in[t_j,t_{j+1})$, set
\[
Y(t,\boldsymbol m)
:=
\Vc_j
\bigl(
    t,m_t;
    m_{t_0},\ldots,m_{t_j}
\bigr)
\]
and
\[
Z(t,x,\boldsymbol m)
:=
\Zc_j
\bigl(
    t,x,m_t;
    m_{t_0},\ldots,m_{t_j}
\bigr).
\]
At the terminal time, set
\[
Y(T,\boldsymbol m)
:=
G(m_{t_0},\ldots,m_{t_k}).
\]
The measurability properties of the parameterized Markovian semigroup imply
that $Y$ and $Z$ are progressively measurable with respect to the grid.

Moreover, for every $j\in\{0,\ldots,k-2\}$,
\[
Y(t_{j+1}-,\boldsymbol m)
=
Y(t_{j+1},\boldsymbol m),
\]
because
\[
\Vc_j
\bigl(
    t_{j+1},m_{t_{j+1}};
    m_{t_0},\ldots,m_{t_j}
\bigr)
=
\Vc_{j+1}
\bigl(
    t_{j+1},m_{t_{j+1}};
    m_{t_0},\ldots,m_{t_{j+1}}
\bigr).
\]

Finally, let $t\in[t_j,t_{j+1})$. Applying the preceding local
representation successively on
\[
[t,t_{j+1}],\quad
[t_{j+1},t_{j+2}],\quad\ldots,\quad
[t_{k-1},t_k],
\]
and using the compatibility relation at each grid time, the intermediate
continuation values telescope. Since
\[
Y(T,\boldsymbol\mu)
=
G(\mu_{t_0},\ldots,\mu_{t_k}),
\]
we obtain
\[
\begin{aligned}
Y(t,\boldsymbol\mu)
&=
G(\mu_{t_0},\ldots,\mu_{t_k})
\\
&\quad
-
\int_t^T
\E\left[
    Z(r,X_r,\boldsymbol\mu)
    \cdot
    B(r,X_r,\boldsymbol\mu)
    \,\middle|\,
    \Gc_r
\right]
\,\mathrm dr
\\
&\quad
-
\int_t^T
\E\left[
    Z(r,X_r,\boldsymbol\mu)
    \,\middle|\,
    \Gc_r
\right]
\cdot
\sigma_\circ\,\mathrm dW^\circ_r.
\end{aligned}
\]
This concludes the proof.
\end{proof}

The proposition shows that discrete path dependence does not require a new
martingale representation theorem. It can instead be incorporated through a
backward dynamic programming procedure: on each interval
$[t_j,t_{j+1}]$, the past values
\[
\mu_{t_0},\ldots,\mu_{t_j}
\]
are frozen as parameters, while the current conditional law $\mu_t$ remains
the Markovian state variable. At the next grid time, the current law is added
to the discrete history, and the construction is iterated backward.

\bibliography{Martingale_regularization}

\end{document}